\documentclass[a4j]{amsart}
\usepackage{amsmath,amssymb}
\usepackage{amsthm}
\pdfoutput=1
\usepackage{graphicx}
\usepackage{bm}
\usepackage{tikz-cd}
\usepackage{mathtools}
\usepackage{comment}
\usepackage{bigdelim, multirow}
\usepackage{latexsym}
\usepackage{stmaryrd}
\usepackage{braket}
\usepackage{color}
\usepackage{setspace}

\newcommand{\Z}{\mathbb{Z}}

\newcommand{\R}{\mathbb{R}}
\newcommand{\C}{\mathbb{C}}
\newcommand{\Q}{\mathbb{Q}}
\newcommand{\T}{\mathbb{T}}

\newcommand{\F}{\mathbb{F}}

\newcommand{\gerg}{\mathfrak{g}}

\newcommand{\gersl}{\mathfrak{sl}}

\newcommand{\gerS}{\mathfrak{S}}

\newcommand{\germ}{\mathfrak{m}}
\newcommand{\gerM}{\mathfrak{M}}
\newcommand{\calY}{\mathcal{Y}}
\newcommand{\calX}{\mathcal{X}}

\newcommand{\calO}{\mathcal{O}}

\newcommand{\calA}{\mathcal{A}}
\newcommand{\calB}{\mathcal{B}}

\newcommand{\calT}{\mathcal{T}}

\newcommand{\calH}{\mathcal{H}}

\newcommand{\om}{\omega}
\newcommand{\Om}{\Omega}

\newcommand{\ga}{\gamma}

\newcommand\vep{\varepsilon}
\newcommand\vphi{\varphi}
\newcommand{\Cstar}{\C^*}

\newcommand\Hom{\mathop{\mathrm{Hom}}\nolimits}

\newcommand\Aut{\mathop{\mathrm{Aut}}\nolimits}

\newcommand\codim{\mathop{\mathrm{codim}}\nolimits}
\newcommand\Coker{\mathop{\mathrm{Coker}}\nolimits}

\newcommand\diag{\mathop{\mathrm{diag}}\nolimits}

\newcommand\Image{\mathop{\mathrm{Im}}\nolimits}

\newcommand\Ker{\mathop{\mathrm{Ker}}\nolimits}

\newcommand\Lie{\mathop{\mathrm{Lie}}\nolimits}
\newcommand\Mat{\mathop{\mathrm{Mat}}\nolimits}
\newcommand\tmin{\mathop{\mathrm{min}}\nolimits}

\newcommand\Pic{\mathop{\mathrm{Pic}}\nolimits}
\newcommand\Proj{\mathop{\mathrm{Proj}}\nolimits}
\newcommand\rank{\mathop{\mathrm{rank}}\nolimits}

\newcommand\Sing{\mathop{\mathrm{Sing}}\nolimits}

\newcommand\Span{\mathop{\mathrm{Span}}\nolimits}
\newcommand\Spec{\mathop{\mathrm{Spec}}\nolimits}
\newcommand\Stab{\mathop{\mathrm{Stab}}\nolimits}

\def\fr<#1/#2>{\frac{#1} {#2}}

\newcommand{\block}[1]{\underbrace{\begin{matrix}1 & \cdots & 1\end{matrix}}_{#1}}
\newcommand{\LargeO}{\text{{\huge{0}}}}

\newcommand{\vncbig}[1]{\multicolumn{2}{c}{$\mbox{{\Huge $#1$}}$}}
\newcommand{\vbig}[1]{\multicolumn{1}{c}{$\mbox{\smash{\Huge $#1$}}$}}
\newcommand{\vcbig}[1]{\multicolumn{2}{c}{$\mbox{\smash{\Huge $#1$}}$}}
\newcommand{\xrightarrowdbl}[2][]{%
  \xrightarrow[#1]{#2}\mathrel{\mkern-14mu}\rightarrow
}

\newtheorem{thm}{Dont use this}[section]
\newtheorem{theorem}[thm]{Theorem}
\newtheorem{proposition}[thm]{Proposition}
\newtheorem{corollary}[thm]{Corollary}
\newtheorem{lemma}[thm]{Lemma}

\newtheorem{question}[thm]{Question}

\newtheorem*{definition*}{Definition}

\newtheorem{definition and lemma}[thm]{Definition $\&$ Lemma}

\newtheorem*{remark*}{Remark}
\theoremstyle{definition}
\newtheorem{definition}[thm]{Definition}
\newtheorem{example}[thm]{Example}

\newtheorem{remark}[thm]{Remark}

\theoremstyle{definition}

\title{%
The universal Poisson deformation of hypertoric varieties and some classification results
}
\author{%
Takahiro Nagaoka%
}
\address[Takahiro Nagaoka]{Department of Mathematics,
	Graduate School of Science,
	Kyoto University,
	Kyoto 606-8502, Japan}
\email{tnagaoka@math.kyoto-u.ac.jp}
\date{}

\begin{document}
\pagestyle{plain}
\begin{abstract}
In this paper, we study and describe the universal Poisson deformation space of hypertoric varieties concretely. In the first application, we show that affine hypertoric varieties as conical symplectic varieties are classified by the associated regular matroids (this is a partial generalization of the result by Arbo and Proudfoot). As a corollary, we obtain a criterion when two quiver varieties whose dimension vector have all coordinates equal to one are isomorphic to each other. Then we describe all 4- and 6-dimensional affine hypertoric varieties as quiver varieties and give some examples of 8-dimensional hypertoric varieties which cannot be raised as such quiver varieties.  In the second application, we compute explicitly the number of all projective crepant resolutions of some 4-dimensional hypertoric varieties by using the combinatorics of hyperplane arrangements.   
\end{abstract} 
\maketitle
\vspace{-33pt}
\section{Introduction}
\begin{spacing}{1}
Hypertoric varieties (or toric hyperk\"{a}hler varieties) were defined as the hyperk\"{a}hler analogues of toric varieties in \cite{BD}, and their topological and geometric properties have been extensively studied by many authors (\cite{HS}, \cite{Kocohomology}, \cite{Kovariation}, \cite{PW}, etc). An affine hypertoric variety has a symplectic form on the smooth locus. Moreover, it is a conical symplectic variety in the sense that it has a good $\C^*$-action for which the symplectic form is homogeneous. 
A conical symplectic variety and its symplectic resolution $\pi: (Y, \om)\to (Y_0, \overline{\om}_0)$ (i.e., $\pi^*\overline{\om}_0$ extends to a symplectic form $\om$ on the whole $Y$, and this is equivalent to be a crepant resolution) also plays an important role in geometric representation theory (\cite{BPW1}, \cite{BLPW2}). 

A symplectic variety has a natural Poisson structure. A Poisson deformation of a symplectic variety is a deformation of the pair of the variety itself and its Poisson structure.
Also, among such  deformations, there is a universal one
(for the detail, see section 3). For a conical symplectic variety $Y_0$ and its symplectic resolution $\pi : Y\to Y_0$, Namikawa showed that there exists a universal Poisson deformation $\calY$ (resp. $\calY_0$) of $Y$ (resp. $Y_0$), and these spaces satisfy the following $\Cstar$-equivariant commutative diagram: 
\begin{equation}\tag{*}\setlength{\abovedisplayskip}{-1pt} 
\begin{tikzcd}[row sep=tiny, column sep=tiny, contains/.style = {draw=none,"\in" description,sloped}, icontains/.style = {draw=none,"\ni" description,sloped}]
&Y\ar[dl, hook']\ar[rr, "\pi"]\ar[dd]&&Y_0\ar[dl, hook'] \ar[dd]\\
\calY\ar[rr, crossing over, "\Pi" near end]\ar[dd, "\overline{\mu}"]&&\calY_0&\\
&0\ar[dl, icontains]\ar[rr, mapsto]&&\overline{0}\ar[dl, icontains]\\
H^2(Y, \C)\ar[rr, "\psi"]&&H^2(Y, \C)/W \ar[from=uu, crossing over, "\overline{\mu}_W" near start]&
\end{tikzcd},\setlength{\belowdisplayskip}{-1pt} 
\end{equation}
where $\psi$ is the Galois cover by a finite group $W$. We call $W$ the Namikawa--Weyl group of $Y_0$. 
The diagram above is determined concretely in the cases of nilpotent orbit closures (\cite{Namdefnil}), Slodowy slices (\cite{LNS}), and quotient singularities (\cite{Bel}).    
Furthermore, it is known that $W$ is closely related to the birational geometry of $Y_0$ (\cite{NamPDbir}). 

In the former part of this paper we give an explicit description of the diagram above for hypertoric varieties. Also, in the latter part we apply this explicit description to classify affine hypertoric varieties and to count the number of different crepant resolutions of affine hypertoric varieties.  
    
\vspace{2pt}

Now we set up notations to state our results precisely. 
Let $A\hspace{-2pt}=\hspace{-2pt}[\bm{a_1}, \ldots, \bm{a_n}]$ be a rank $d$ unimodular matrix and take $B^T\hspace{-2pt}=\hspace{-2pt}[\bm{b_1}, \ldots ,\bm{b_n}]$ so that  the following is exact: 
\begin{equation}\tag{**}\setlength{\abovedisplayskip}{0pt}
\begin{tikzcd}
0\ar[r]&\Z^{n-d}\ar[r, "{B}"] &\Z^{n}\ar[r, "{A}"]&\Z^d\ar[r]&0         
\end{tikzcd}.\setlength{\belowdisplayskip}{0pt}
\end{equation}
Then by using the natural Hamiltonian $\T_\C^d$-action on $\C^{2n}$ induced from  $A^T: \T_\C^d\hookrightarrow \T_\C^n$, we define a hypertoric variety $Y(A, \alpha):=\mu^{-1}(0)/\hspace{-3pt}/_\alpha\T_\C^d$ as the GIT quotient, where $\alpha\in\Z^d$ is a GIT parameter, and $\mu: \C^{2n}\to\C^d$ is the $\T_\C^d$-invariant moment map. 
By definition, we have a natural projective morphism $\pi:Y(A, \alpha)\to Y(A, 0)$. Then, $Y(A, 0)$ is a conical symplectic variety, and for generic $\alpha$, $\pi$ is a projective symplectic resolution. 

In section 3, we construct the diagram (*) for $\pi: Y(A, \alpha)\to Y(A, 0)$. First, we consider the Poisson variety $X(A, \alpha):=\C^{2n}/\hspace{-3pt}/_\alpha\T_\C^d$ (called the Lawrence toric variety) and a natural projective morphism $\Pi: X(A, \alpha) \to X(A, 0)$. Then, the induced map $\overline{\mu}_\alpha: X(A, \alpha)\to\C^d$ gives a Poisson deformation of $Y(A, \alpha)$ and satisfy the following $\Cstar$-equivariant diagram: 
\[\setlength{\abovedisplayskip}{0pt}
\begin{tikzcd}[row sep=tiny, column sep=tiny, contains/.style = {draw=none,"\in" description,sloped}, icontains/.style = {draw=none,"\ni" description,sloped}]
&Y(A, \alpha)\ar[dl, hook']\ar[rr, "\pi"]\ar[dd]&&Y(A, 0)\ar[dl, hook'] \ar[dd]\\
X(A, \alpha)\ar[rr, crossing over, "\Pi" near end]\ar[dd, "\overline{\mu}_\alpha"]&&X(A, 0)&\\
&0\ar[dl, contains]\ar[rr, mapsto]&&0\ar[dl, contains]\\
\C^d\ar[rr, equal]\ar[ur, icontains]&&\C^d \ar[from=uu, crossing over, "\overline{\mu}_0"near start]\ar[ur, icontains]&
\end{tikzcd}\setlength{\belowdisplayskip}{0pt}.\] 
Then, \hspace{-1pt}as\hspace{-1pt} is \hspace{-1pt}well-known \hspace{-1pt}(but \hspace{-1pt}the \hspace{-1pt}proof \hspace{-1pt}doesn't \hspace{-1pt}appear \hspace{-1pt}explicitly \hspace{-1pt}in \hspace{-1pt}literature), \hspace{-1pt}$\overline{\mu}_\alpha:X(A, \alpha) \hspace{-2pt}\to\hspace{-2pt} \C^d$ \hspace{-1pt}is \hspace{-1pt}the \hspace{-1pt}universal \hspace{-1pt}Poisson \hspace{-1pt}deformation \hspace{-1pt}of $Y(A, \alpha)$ (see \hspace{-1pt}Proposition \ref{prop:UPDforsmooth}). 

Next, to describe the universal Poisson deformation of $Y(A, 0)$, we determine the action of the Namikawa--Weyl group $W$ explicitly. We assume that $B$ is in the form of \renewcommand{\arraystretch}{0.5}\arraycolsep=1.4pt $B^T=[\bm{b^{(1)}} \cdots \bm{b^{(1)}} \cdots \bm{b^{(s)}}\cdots \bm{b^{(s)}}]$, where $\bm{b^{(k_1)}}\neq\pm\bm{b^{(k_2)}}$ if $k_1\neq k_2$ (we can always assume this). 
\renewcommand{\arraystretch}{1}\arraycolsep=2pt
Then, we prove that the Namikawa--Weyl group of $Y(A, 0)$ is $W_B:=\gerS_{\ell_1}\times\cdots\times\gerS_{\ell_s}$, where $\ell_i$ is the multiplicity of $\bm{b^{(i)}}$. Moreover, by constructing a suitable $W_B$-action (cf.\ Proposition \ref{prop:HT-NWgrp}) on $X(A, 0)$ and $\C^d$ such that $\overline{\mu}_0:X(A, 0)\to\C^d$ is $W_B$-equivariant, we describe the universal Poisson deformation space of $Y(A, 0)$ as the following.  
This result was mentioned in \cite{BPW1} without proof. 
\begin{theorem}{\rm (cf.\ Theorem \ref{thm:Mainthm})}\label{thm:iMainthm}
In the setting above, if each row $\bm{b^{(i)}}$ of $B$ is not a zero vector, the diagram (*) for $\pi : Y(A, \alpha) \to Y(A, 0)$ will be the following:  
\[\setlength{\abovedisplayskip}{-1pt}
\begin{tikzcd}[row sep=tiny, column sep=tiny, contains/.style = {draw=none,"\in" description,sloped}, icontains/.style = {draw=none,"\ni" description,sloped}]
&Y(A, \alpha)\ar[dl, hook']\ar[rr, "\pi"]\ar[dd]&&Y(A, 0)\ar[dl, hook'] \ar[dd]\\
X(A, \alpha)\ar[rr, crossing over, "\Pi_{W_B}" near end]\ar[dd, "\overline{\mu}_\alpha"]&&X(A, 0)/W_B&\\
&0\ar[rr, mapsto]&&\overline{0}\\
\C^d\ar[rr, "\psi"]\ar[ur, icontains]&&\C^d/W_B\ar[ur, icontains] \ar[from=uu, crossing over, "\overline{\mu}_{W_B}" near start]&
\end{tikzcd},\setlength{\belowdisplayskip}{-1pt}\]
where $\Pi_{W_B}$ is the composition of $\Pi$ and the quotient map of $X(A, 0)$ by $W_B$. 
\end{theorem}


In section 4, as applications of Theorem \ref{thm:iMainthm}, we will {prove} some classification results on affine hypertoric varieties. First, from the work in \cite{AP}, it is easily deduced that the $\Cstar\times\T_\C^{n-d}$-equivariant isomorphism classes of affine hypertoric varieties ({as Poisson varieties}) are bijective to the isomorphism classes of the regular matroids, {$\T_\C^{n-d}$-action is the remaining hamiltonian torus $\T_\C^n/\T_\C^d$-action.} Then, our result is the following. 

\begin{theorem}\label{thm:introclassification}{\rm (cf.\ Theorem \ref{thm:classification})}
Let $A$ and $A'$ be the same size unimodular matrices. Then the following are equivalent. 
\begin{itemize}
\item[\hspace{-5pt}(i)]$Y(A, 0)$ is isomorphic to $Y(A', 0)$ as conical symplectic varieties. 
\item[\hspace{-5pt}(ii)]$Y(A, 0)$ is $\Cstar\hspace{-1pt}\times\hspace{-1pt}\T_\C^{n-d}$-equivariant isomorphic to $Y(A', 0)$ as Poisson varieties. 
\end{itemize} 
In particular, the isomorphism classes of affine hypertoric varieties as conical symplectic varieties are bijective to the isomorphism classes of the regular matroids.
\end{theorem}
Our proof is based on Theorem \ref{thm:iMainthm}. By the universality, if $Y(A, 0)$ and $Y(A', 0)$ are isomorphic to each other as conical symplectic varieties, we can show {that} $X(A, 0)$ and $X(A', 0)$ are $\Cstar$-equivariant isomorphic as Poisson varieties. By \cite{Ber}, an isomorphism between two toric varieties {can be replaced by a maximal torus $\T^{2n-d}_\C$-equivariant isomorphism. Moreover, we show that this isomorphism can be taken as $\Cstar\times\T_\C^{2n-d}$-equivariant one preserving Poisson structures and commuting with moment maps.} In particular, we can prove that $Y(A, 0)$ and $Y(A', 0)$ are $\T_\C^{n-d}$-equivariant isomorphic. 



For a finite graph $G$, the hypertoric variety $Y(A_G, 0)$ (called toric quiver varieties) associated to the incidence matrix $A_G$ of $G$ is called a toric quiver variety. This is the same as the Nakajima quiver variety whose dimension vector has all coordinates equal to one (see Example \ref{eg:quiver}). As a first corollary of Theorem \ref{thm:introclassification}, we obtain a combinatorial criterion that two affine toric quiver varieties are mutually isomorphic as conical symplectic varieties. 
\begin{corollary}(cf.\ Corollary \ref{cor:quiver})
Let $G_1$ and $G_2$ be two finite graphs (without isolated vertices). Then, $Y(A_{G_1}, 0)\cong Y(A_{G_2}, 0)$ {as conical symplectic varieties} if and only if $G_2$ can be transformed into $G_1$ by a sequence of the following operations: {\rm (i)} A vertex identification, a vertex cleaving, and {\rm (ii)} the Whitney twist (for the detail, see Corollary \ref{cor:quiver}).

\end{corollary} 
  
As a second corollary, we classify 4- and 6- dimensional affine hypertoric varieties concretely. By Theorem \ref{thm:introclassification}, this is equivalent to classify regular matroids of $\rank$ 2 and 3. As a result, these affine hypertoric varieties are obtained as toric quiver varieties (we describe the corresponding quivers, and give some 8-dimensional examples which cannot be raised as any toric quiver varieties in Remark \ref{rem:quiver2}). In the 4-dimensional case, we can describe more explicitly as the following:

\begin{theorem}{\rm (cf.\ Theorem \ref{thm:4dimension})}
Each 4-dimensional affine hypertoric variety $Y(A, 0)$ is isomorphic to exactly one of the following as conical symplectic varieties:
\begin{itemize}
\item[(i)]$S_{A_{\ell_1-1}}\times S_{A_{\ell_2-1}}$,
\item[(ii)]$\overline{{\calO}^{\tmin}}(\{\ell_1, \ell_2, \ell_3\}):=\Set{\begin{pmatrix}u_1&x_{12}&x_{13}\\y_{12}&u_2&x_{23}\\y_{13}&y_{23}&u_3\end{pmatrix} \in \gersl_3 \ | \ \text{{\rm All 2 by 2 minors of}} \begin{pmatrix}u_1^{\ell_1}&x_{12}&x_{13}\\y_{12}&u_2^{\ell_2}&x_{23}\\y_{13}&y_{23}&u_3^{\ell_3}\end{pmatrix}=0}$,
\end{itemize}
where $S_{A_{\ell-1}}$ is the $A_{\ell-1}$ type surface singularity. 
Furthermore, the Namikawa--Weyl group $W$ is given by {\rm (i)} $W=\gerS_{\ell_1}\times\gerS_{\ell_2}$, {\rm (ii) }$W=\gerS_{\ell_1}\times\gerS_{\ell_2}\times\gerS_{\ell_3}$.
\end{theorem} 

As a related result, we can also characterize nilpotent orbit closures among affine hypertoric varieties (cf.\ Theorem \ref{thm:char_nilpotent}).


Finally, in section 5, we will consider how many projective crepant resolutions $Y(A, 0)$ admits. By combining known results, we note all such resolutions of each $Y(A, 0)$ can be obtained by the variations of the GIT parameter $\alpha$. Then, we can count the number of all projective crepant resolutions in terms of the associated hyperplane arrangement $\calH_A$ (see section 5) as the following.   
\begin{lemma}{\rm (cf.\ Corollary \ref{cor:counting})}
For $Y(A, 0)$, the number of its projective crepant resolutions is $\fr<r(\calH_A)/|W_B|>$, where $r(\calH_A)$ is the number of chambers of $\calH_A$.
\end{lemma}
Then, we focus on the case of the (non-trivial) 4-dimensional affine hypertoric variety $\overline{\calO^{\tmin}}(\{\ell_1, \ell_2, \ell_3\})$. Since we can show {that} the associated $\calH_A$ is essentially the same as the one considered in \cite{ER}, we determine the number $\fr<r(\calH_A)/|W_B|>$ when $\ell_3=1$ or $2$ explicitly (cf.\ Corollary \ref{cor:counting_4dim}). 
\end{spacing}
\section*{Acknowledgements}
The author wishes to express his gratitude to his supervisor Yoshinori Namikawa for suggesting {him} the problem and for many stimulating discussions. {He also wishes to thank Nicholas Proudfoot for answering some questions.} 
He is also grateful to Akiyoshi Tsuchiya for his several helpful suggestions concerning Macaulay2. {The author expresses his thanks to Makoto Enokizono for pointing out a gap in the previous proof of Theorem \ref{thm:classification}. }
This paper is based on the author's master thesis. 
\section{Preliminary - definition and basic results on hypertoric varieties}

In this section, we review the definition of (conical) symplectic varieties and hypertoric varieties, and we recall some basic properties of them. 
A pair $(Y_0, \overline{\om}_0)$ of a normal algebraic variety $Y_0$ and {a} 2-form $\overline{\om}_0$ on the smooth locus $(Y_0)_{\text{sm}}$ is called a {\it symplectic variety} if $\overline{\om}_0$ is symplectic and there exists (or equivalently, for any) a resolution $\pi: Y\to Y_0$ such that the pull-back of $\overline{\om}_0$ by $\pi$ extends to {a} holomorphic 2-form $\om$ on $Y$. Moreover, a resolution $\pi: Y\to Y_0$ is called {\it symplectic} if $\om$ is also symplectic. These definitions are due to Beauville \cite{Bea}, and it is known that for a symplectic variety $Y_0$, a resolution $\pi: Y\to Y_0$ is a symplectic resolution if and only if $\pi$ is a crepant resolution, i.e., $\pi^*K_{Y_0}=K_Y$, where $K_Y$ (resp. $K_{Y_0}$) denotes the canonical sheaf of $Y$ (resp. $Y_0$). In this paper, we only consider projective crepant resolutions, so if we simply say crepant resolutions, it means projective ones. 

Now, we define conical symplectic varieties as the following: 
\begin{definition}\label{def:conical}
An affine symplectic variety $(Y_0\hspace{-2pt}=\hspace{-2pt}\Spec R, \ \overline{\om}_0)$ with $\Cstar$-action (called conical $\Cstar$-action) is called a {\it conical symplectic variety} if it satisfies the following: 
\begin{itemize}
\item[(i)]The grading induced from the $\Cstar$-action to the coordinate ring $R$ is positive, i.e., $R=\bigoplus_{i\geq0}{R_i}$ and $R_0=\C$. 
\item[(ii)]$\overline{\om}_0$ is {\it homogeneous} with respect to the $\Cstar$-action, i.e., there exists $\ell\in\Z$ (the {\it weight} of $\overline{\om}_0$) such that $t^*\overline{\om}_0=t^{\ell}\overline{\om}_0 \ (t\in\Cstar)$. 
\end{itemize}
For a given two conical symplectic varieties, we say that they are isomorphic as conical symplectic varieties if there exists an $\Cstar$-equivariant isomorphism between them preserving symplectic structures.   
\end{definition}  
\begin{remark}
As noted in \cite[Lemma 2.2]{NamEq}, the weight $\ell$ is always positive. 
\end{remark}

There are many examples of conical symplectic varieties, for example, the (normalization of) nilpotent orbit closures in semisimple Lie algebras, the Slodowy slices, quiver varieties, hypertoric varieties, symplectic quotient singularities, and so on. 
Next, we define Poisson varieties (more generally, Poisson $S$-schemes), which include symplectic varieties as special cases. 
\begin{definition}Let $S$ be a scheme. A pair $(\calY, \{-, -\})$ of a $S$-scheme and a skew-symmetric $\calO_S$-bilinear {morphism} $\{-, -\}: \calO_{\calY}\times\calO_{\calY}\to\calO_{\calY}$ is a {\it Poisson $S$-scheme} if it satisfies the following: 
\begin{itemize}
\item[(1)]$\{f, gh\}=\{f, g\}h+\{f, h\}g$,
\item[(2)]$\{f, \{g, h\}\}+\{g, \{h, f\}\}+\{h, \{f, g\}\}=0$.
\end{itemize}
In particular, if $S=\Spec\C$ and $\calY$ is an algebraic variety, then we will call $\calY$ a {\it Poisson variety}. 
\end{definition}
Below, for an algebraic variety $Z$, we denote its smooth locus by $Z_{\text{sm}}$. For a symplectic variety $(Y_0, \overline{\om}_0)$, one can consider a homomorphism $H : \calO_{(Y_0)_{\text{sm}}}\to \calT_{(Y_0)_{\text{sm}}}$ as $\calO_{(Y_0)_{\text{sm}}}$-module which is defined as the composition of the exterior derivation $d: \calO_{(Y_0)_{\text{sm}}}\to \Om^1_{(Y_0)_{\text{sm}}}$ and the inverse of the isomorphism $\calT_{(Y_0)_{\text{sm}}} \to \Om^1_{(Y_0)_{\text{sm}}} : v\mapsto \overline{\om}_0(v, -)$, where $\Om^1_{(Y_0)_{\text{sm}}}$ is the sheaf of 1-forms on $(Y_0)_{\text{sm}}$, and $\calT_{(Y_0)_{\text{sm}}}$ is its dual sheaf, so-called the {tangent sheaf}. Then, a natural Poisson structure on $(Y_0)_{\text{sm}}$ is defined as $\{f, g\}:=\overline{\om}_0(H_f, H_g)$, where $H_f:=H(f)$ ($f\in\calO_{(Y_0)_{\text{sm}}}$). By the normality of $Y_0$, this Poisson structure uniquely extends to the one on $\calO_{Y_0}$. In particular, for a conical symplectic variety $Y_0=\Spec R$ with weight $\ell$, by definition, $\{f, g\}\in R_{i+j-\ell} \ (f\in R_i, g\in R_j)$. In section 3, we discuss deformations of this Poisson structure of a symplectic variety.

Now we move to define hypertoric varieties and related notions. Our notation is the same as \cite{HS}. Let $A: \Z^n\xrightarrowdbl{}N\cong\Z^d$ be a surjective linear map (in this paper, we almost always fix a basis of $N$, identify $N$ with $\Z^d$ and consider $A=[\bm{a_1}, \ldots, \bm{a_n}]$ as a $d\times n$-matrix). Then, take $B^T=[\bm{b_1}, \ldots ,\bm{b_n}]\in \Mat_{(n-d)\times n}(\Z)$ as the following is exact: 
\[\begin{tikzcd}
0\arrow{r}&\Z^{n-d}\arrow{r}{B} &\Z^{n}\arrow{r}{A}&N\cong\Z^d\arrow{r}&0         
\end{tikzcd}.\]
The configuration $\{\bm{b_1}, \ldots ,\bm{b_n}\}$ in $\Z^{n-d}$ is called a {\it Gale dual} of $\{\bm{a_1}, \ldots ,\bm{a_n}\}$. 
By applying $\Hom(-, \Cstar)$ to the above exact sequence, we obtain the following exact sequence of algebraic tori. 
\[\begin{tikzcd}
1\ar[r]&\T_\C^d\ar[r, "{A^T}"]&\T_\C^n\ar[r, "{B^T}"]&\T_\C^{n-d}\ar[r]&1
\end{tikzcd},\] 
where $\T_\C^k:=(\Cstar)^k$. Then, through the embedding $\T_\C^d\overset{A^T}{\hookrightarrow}\T_\C^n$ and the natural action $\T_\C^n\curvearrowright(T^*\C^n, \omega_\C)=(\C^{2n}, \sum_{j=1}^n{dz_j\wedge dw_j})$, we obtain a hamiltonian $\T_\C^d$-action on $(\C^{2n}, \omega_\C)$ (for the definition of the hamiltonian action, we refer \cite[Definition 9.43]{Kir}). More explicitly, this action is described as the following: 
\[\bm{t}\cdot(z_1, \ldots ,z_n, w_1, \ldots ,w_n):=(\bm{t}^{\bm{a_1}}z_1, \ldots ,\bm{t}^{\bm{a_n}}z_n, \bm{t}^{\bm{-a_1}}w_1, \ldots ,\bm{t}^{\bm{-a_n}}w_n),\]
where $\bm{t}^{\bm{a_j}}:=t_1^{a_{1j}}\cdots t_d^{a_{dj}}$. 
Since this action is hamiltonian, we have the $\T_\C^d$-invariant moment map $\mu: \C^{2n}\to(\Lie(\T_\C^d))^{{*}}=\C^d$ (more strongly, $\mu$ is $\T_\C^n$-invariant), and in this case, we can describe it as the following: 
\[\mu(\bm{z}, \bm{w})=\sum_{j=1}^nz_jw_j\bm{a_j}.\]
Then, the Lawrence toric variety (resp. hypertoric variety) is defined as the GIT quotient of $\C^{2n}$ (resp. $\mu^{-1}({\xi})$ ({$\xi\in\C^d$})) by $\T_\C^d$, and actually we can describe this explicitly as the following (cf.\ \cite{HS}). 
For $\alpha\in\Z^d=\Hom(\T_\C^d, \Cstar)$ and the coordinate ring $\C[\bm{z}, \bm{w}]$ (resp. $\C[\mu^{-1}(\xi)]$) of $\C^{2n}$ (resp. $\mu^{-1}(\xi)$), we set 
\[\C[M]^{\T_\C^d, \alpha}:=\{f\in\C[M] \ | \ f(\bm{t}\cdot m)=\alpha(\bm{t})f(m) \ \forall \bm{t}\in\T_\C^d\},\]
where $M=\C^{2n}$ or $\mu^{-1}(\xi)$. 

\begin{definition}For $\alpha\in\Z^d=\Hom(\T_\C^d, \Cstar)$, we consider a graded algebra $\bigoplus_{k\in\Z_{\geq0}}{\C[M]^{\T_\C^d, k\alpha}}$. Then, we define the {\it Lawrence toric variety} $X(A, \alpha)$ and $Y(A, (\alpha, \xi))$ as the following: 
\[X(A, \alpha):=\Proj\left(\bigoplus_{k\in\Z_{\geq0}}{\C[\bm{z}, \bm{w}]^{\T_\C^d, k\alpha}}\right), \ \ \ Y(A, (\alpha, \xi)):=\Proj\left(\bigoplus_{k\in\Z_{\geq0}}{\C[\mu^{-1}(\xi)]^{\T_\C^d, k\alpha}}\right).\]
Since we will mainly consider the case of $\xi=0$, so we often write $Y(A, \alpha)$ for $Y(A, (\alpha, 0))$. We call $Y(A, \alpha)$ a {\it hypertoric variety}. 
\end{definition}

\begin{remark}\label{rem:comm diagram}
Since in the case of $\alpha=0$, we have $\Proj\left(\bigoplus_{k\in\Z_{\geq0}}{\C[M]^{\T_\C^d, k\alpha}}\right)=\Spec \C[M]^{\T_\C^d}$, so for any $\alpha$, we have a natural projective morphism $\Pi: X(A, \alpha)\to X(A, 0)$ (resp. $\pi: Y(A, \alpha) \to Y(A, 0)$). Furthermore, as $\mu$ is a $\T_\C^d$-invariant map, we have the commutative diagram below. Also, we consider a natural $\Cstar$-action on $\C^{2n}$ such that $s\cdot(z_1, \ldots ,z_n, w_1, \ldots ,w_n)=(s^{-1}z_1, \ldots ,s^{-1}z_n, s^{-1}w_1, \ldots ,s^{-1}w_n)$. Since this action commutes with $\T_\C^d$-action, the following diagram is $\Cstar$-equivariant with respect to the induced $\Cstar$-action, where $\Cstar$ acts on $\C^d$ as $s\cdot v=s^{-2}v$ ($v\in\C^d$).   
\[
\begin{tikzcd}[row sep=tiny, column sep=tiny, contains/.style = {draw=none,"\in" description,sloped}, icontains/.style = {draw=none,"\ni" description,sloped}]
&Y(A, \alpha)\ar[dl, hook']\ar[rr, "\pi"]\ar[dd]&&Y(A, 0)\ar[dl, hook'] \ar[dd]\\
X(A, \alpha)\ar[rr, crossing over, "\Pi" near end]\ar[dd, "\overline{\mu}_\alpha"]&&X(A, 0)&\\
&0\ar[dl, contains]\ar[rr, mapsto]&&0\ar[dl, contains]\\
\C^d\ar[rr, equal]\ar[ur, icontains]&&\C^d \ar[from=uu, crossing over, "\overline{\mu}_0"near start]\ar[ur, icontains]&
\end{tikzcd}
\]  
\end{remark}
Although our definition of Lawrence toric varieties and hypertoric varieties is the abstract one, we can also define them geometrically as the categorical quotient of the $\alpha$-(semi)stable subset of $M=\C^{2n}$ (resp. $\mu^{-1}(\xi)$) in the sense of geometric invariant theory. An element $p\in M$ is called {\it $\alpha$-semistable} if there exists some $k>0$ and $f\in\C[M]^{\T_\C^d, k\alpha}$ such that $f(p)\neq0$, and we denote the set of $\alpha$-semistable elements by $M^{\alpha-ss}$. 
Then, as well-known in geometric invariant theory, we have $X(A, \alpha)=(\C^{2n})^{\alpha-ss}/\hspace{-3pt}/\T_\C^d$ (resp. $Y(A, (\alpha, \xi))=\mu^{-1}(\xi)^{\alpha-ss}/\hspace{-3pt}/\T_\C^d$), where $/\hspace{-3pt}/$ denotes the categorical quotient. 
Also, an element $p\in M^{\alpha-ss}$ is called {\it $\alpha$-stable} if the stabilizer group of $\T_\C^d$ at $p$ is finite and $\T_\C^d\cdot p\subseteq M^{\alpha-ss}$ is closed. We denote the set of $\alpha$-stable elements by $M^{\alpha-st}$.     

As noted in \cite[Lemma 3.4]{Kosurvey} (actually, by some easy computations), we describe the $\alpha$-semistable set as the following:  
\begin{equation*}\tag{*}
(\C^{2n})^{\alpha-ss}=\Set{(\bm{z}, \bm{w}) \in \C^{2n} \ | \ \alpha \in \sum_{i: z_i \neq0}{\Q_{\geq0}\bm{a_i}} + \sum_{i:w_i\neq0}{\Q_{\geq0}(-\bm{a_i})}}.
\end{equation*}
Also, we have $\mu^{-1}(\xi)^{\alpha-ss}=\mu^{-1}(\xi)\cap(\C^{2n})^{\alpha-ss}$. 

To consider when $(\C^{2n})^{\alpha-ss}$ and $(\C^{2n})^{\alpha-st}$ coincides, we describe the condition that the stabilizer group at a point $(\bm{z}, \bm{w})\in\C^{2n}$ of $\T_\C^d$-action is finite (resp. trivial). First, we remark the following easy lemma.  
\begin{lemma}For a subset $J\subseteq [n]:=\{1, 2, \ldots ,n\}$, we consider the subtorus $\T_\C^{|J|}\subseteq\T_\C^{n}$ corresponding  to the natural projection $\Z^n\xrightarrowdbl{} \Z^{|J|}$. Denote the submatrix of $A$ which consists of column vectors corresponding to $J$ by $A_J$. Then we have the following:  
\begin{itemize}
\item[(1)] $\T_\C^d\xrightarrow{A_J^T}\T_\C^{|J|}$ is injective if and only if $\Z^{|J|}\xrightarrow{A_J}\Z^d$ is surjective.
\item[(2)]The kernel of $\T_\C^d\xrightarrow{A_J^T}\T_\C^{|J|}$ is finite if and only if $\Q^{|J|}\xrightarrow{A_J\otimes_\Z\Q}\Q^d$ is surjective. 
\end{itemize}
\end{lemma}
Now, for $(\bm{z}, \bm{w}) \in \C^{2n}$, we set $J_{\bm{z}, \bm{w}}:=\{j \in [n] \ | \ z_j\neq0 \ \text{or}\ w_j\neq0\}$. Then, the stabilizer group {$\Stab_{\bm{z}, \bm{w}}\T_\C^d$} of the action $\T_\C^d$ on $\C^{2n}$ at $(\bm{z}, \bm{w})$ is given by 
\begin{align*}
{\Stab_{\bm{z}, \bm{w}}\T_\C^d}=\Ker(\T_\C^d \xrightarrow{A_{J_{\bm{z}, \bm{w}}}^T}\T_\C^{|J_{\bm{z}, \bm{w}}|}). 
\end{align*}

\begin{corollary}\label{cor:stabilizer} In the setting above, 
\begin{itemize}
\item[(1)]{$\Stab_{\bm{z}, \bm{w}}\T_\C^d$} is finite if and only if $\sum_{j\in J_{\bm{z}, \bm{w}}}{\Q\bm{a_j}}=\Q^d$. 
\item[(2)]{$\Stab_{\bm{z}, \bm{w}}\T_\C^d=1$} if and only if $\sum_{j\in J_{\bm{z}, \bm{w}}}{\Z\bm{a_j}}=\Z^d$. 
\end{itemize}
\end{corollary}

To consider when $\T_\C^d$ acts on $M^{\alpha-st}$ freely, we define the following. 
\begin{definition}\label{def:unimodular}
In the setting above, $A$ is {\it unimodular} if all $d\times d$-minors of $A$ are 0 or $\pm1$. Clearly, this is equivalent to that $B$ is unimodular (cf.\ \cite[Definition 2.8]{HS}). 
\end{definition}
\begin{remark}\label{rem:unimodular}
By definition, for $J\subseteq[n]$ and a unimodular matrix $A$, if $\sum_{j\in J}{\Q\bm{a_j}}=\Q^d$, then $\sum_{j\in J}{\Z\bm{a_j}}=\Z^d$. 
\end{remark}
In this paper, we always assume that $A$ is a unimodular matrix {unless otherwise stated}. 
Next, consider the hyperplane arrangement $\calH_A$ in $\R^d$ defined as the following. 
\[\calH_A:=\{H{\subset \R^d} \ | \ H \ \text{is generated by some $\bm{a_j}$'s and $\codim H=1$}\}.\]
We say $\alpha\in\Z^d$ is {\it generic} if $\alpha\notin\bigcup_{H\in\calH_A}{H}$. As we see in the following, for a generic $\alpha$, we have $M^{\alpha-ss}=M^{\alpha-st}$, so the categorical quotient $M^{\alpha-ss}/\hspace{-3pt}/\T_\C^d$ is the geometric quotient $M/\T_\C^d$ (cf.\ \cite[Proposition 3.6 (4)]{Kocohomology}).

\begin{lemma}\label{lem:stable set HT}
For any $\alpha \in \Z^d$ and $\xi \in\C^d$, we have $(\mu^{-1}(\xi))^{\alpha-ss}\neq\phi$. Moreover, if $\alpha$ is generic, then $(\mu^{-1}(\xi))^{\alpha-ss}=(\mu^{-1}(\xi))^{\alpha-st}$ and  $(\C^{2n})^{\alpha-ss}=(\C^{2n})^{\alpha-st}$. 
Furthermore, 
the $\T_\C^d$-action on $(\C^{2n})^{\alpha-st}$ will be free. 
\end{lemma}
\begin{proof}
First, we will show that $(\mu^{-1}(\xi))^{\alpha-ss}\neq\phi$ for any $\alpha \in \Z^d$ and $\xi \in\C^d$. Note that we can express $\xi=\sum_{j\in J}{c_j\bm{a_j}}$, where $J:=\{j \in[n] \ | \ c_j\neq0\}$, and $\alpha=\sum_{j\in J^+}{c'_j\bm{a_j}}+\sum_{j\in J^-}{c'_j\bm{a_j}}$, where $J^\pm:=\{j\in[n] \ | \ \pm c'_j>0\}$ respectively. Then one can easily show that there exists $(\bm{z}, \bm{w})\in \C^{2n}$ such that $z_jw_j=c_j \ (j \in J)$, $z_j\neq0, w_j=0 \ (j \in J^+\cap J^c)$, and $z_j=0, w_j\neq0 \ (j \in J^-\cap J^c)$, where $J^c:=[n]-J$. From (*), this $(\bm{z}, \bm{w})$ is an element of $(\mu^{-1}(\xi))^{\alpha-ss}$. Hence, we have $(\mu^{-1}(\xi))^{\alpha-ss}\neq\phi$.   

Next, we assume that $\alpha$ is generic. {W}e want to show that $(\C^{2n})^{\alpha-ss}=(\C^{2n})^{\alpha-st}$ and $(\mu^{-1}(\xi))^{\alpha-ss}=(\mu^{-1}(\xi))^{\alpha-st}$. Since $\mu$ is $\T_\C^d$-invariant, so it is sufficient to prove $(\C^{2n})^{\alpha-ss}=(\C^{2n})^{\alpha-st}$. Then we only have to show that the stabilizer group of $\T_\C^d$ at each point in $(\C^{2n})^{\alpha-ss}$ is finite. In fact, if we can prove this but for some $x\in(\C^{2n})^{\alpha-ss}$, the $\T_\C^d$-orbit $\T_\C^d\cdot x$ is not closed, then we can take an element $y\in\overline{\T_\C^d\cdot x}-\T_\C^d\cdot x$, where $\overline{\T_\C^d\cdot x}$ is the Zariski closure. Also, by the closed orbit lemma (\cite[section 8.3 Proposition]{Hum}), we have $\dim \T_\C^d\cdot y<\dim \T_\C^d\cdot x$. However, by the assumption, $\dim {\Stab_{x}\T_\C^d}=\dim {\Stab_{y}\T_\C^d}=0$, so it should be that $\dim \T_\C^d\cdot y=\dim \T_\C^d\cdot x=\dim \T_\C^d$. This is a contradiction.  

To show that the stabilizer group of $\T_\C^d$ at each point in $(\C^{2n})^{\alpha-ss}$ is finite, we note that $\alpha\in\sum_{j\in J_{\bm{z}, \bm{w}}}{\Q\bm{a_j}}$ holds for $(\bm{z}, \bm{w})\in(\C^{2n})^{\alpha-ss}$ by (*). 
On the other hand, $\alpha$ is generic, so $\alpha$ cannot be in any hyperplanes in $\calH_A$. Since this implies that $\sum_{j\in J_{\bm{z}, \bm{w}}}{\Q\bm{a_j}}=\Q^d$, the stabilizer group at each point in $(\C^{2n})^{\alpha-ss}$ is finite by Corollary \ref{cor:stabilizer} (1). Moreover, $A$ is unimodular, {so} by Remark \ref{rem:unimodular}, we have $\sum_{j\in J_{\bm{z}, \bm{w}}}{\Z\bm{a_j}}=\Z^d$. Thus the stabilizer group at each point in $(\C^{2n})^{\alpha-ss}$ is trivial by Corollary \ref{cor:stabilizer} (2).
\end{proof}
Then, by the symplectic quotient construction, we have the following (see for example, \cite[Theorem 9.53]{Kir}). 
\begin{corollary}\label{cor:hypertoricsymplectic}
We assume that 
$\alpha\in\Z^d$ is generic. Then  $X(A, \alpha)$ is a $2n-d$ dimensional smooth Poisson variety, and for any $\xi\in\C^d$, $Y(A, (\alpha, \xi))$ is a $2n-2d$  dimensional smooth symplectic variety. 
\end{corollary}

Next, we will show that $Y(A, (0, \xi))$ is a symplectic variety and $\pi_\xi : Y(A, (\alpha, \xi)) \to Y(A, (0, \xi))$ is a projective symplectic resolution. 
\begin{lemma}\label{lem:CI}
For any $\xi\in\C^d$, ${\mu}^{-1}(\xi)$ is a $2n-d$ dimensional complete intersection in $\C^{2n}$. 
\qed\end{lemma}
\begin{proof}
Note that the moment map $\mu$ is the composition $\C^{2n}\xrightarrow{\Psi}\C^n\xrightarrow{A}\C^d$, where $\Psi(\bm{z}, \bm{w}):=\sum_{j}z_jw_j\bm{e_j}$ and $\bm{e_j}$ is the $j$-th standard unit vector. And, it can be easily showed that $\Psi$ is a flat morphism. 
Then, by $\dim A^{-1}(\xi)=\dim \Ker A=n-d$ and the basic property of flat morphisms (\cite[III Corollary 9.6]{Hartshorne}), we can show that the dimension of all irreducible components of $\mu^{-1}(\xi)=\Phi^{-1}(A^{-1}(\xi))$ is $(n-d)+n=2n-d$. Since $\mu^{-1}(\xi)$ is defined in $\C^{2n}$ as the vanishing locus of $d$ polynomials, we can deduce that ${\mu}^{-1}(\xi)$ is a $2n-d$ dimensional complete intersection in $\C^{2n}$. 
\end{proof}

\begin{proposition}\label{prop:birational}
We assume that $\bm{b_j}\neq0$ for every $j=1, \ldots, n$. Then, $Y(A, (0, \xi))$ is a normal (irreducible) variety for any $\xi\in\C^d$. Also, $({\mu}^{-1}(\xi))^{0-st}\neq \phi$, in particular, $\dim Y(A, (0, \xi))=2n-2d$. Moreover, if $\alpha$ is generic, then $\pi:Y(A, (\alpha, \xi))\to Y(A, (0, \xi))$ will be birational. 
\qed\end{proposition}
\begin{proof}
First, under the assumption, we will show that the singular locus $({\mu}^{-1}(\xi))_{\Sing}$ is at least codimension 2 in ${\mu}^{-1}(\xi)$.
Note that the tangent map of $\mu$ at $(\bm{z}, \bm{w})$\\ 
$d\mu_{(\bm{z}, \bm{w})}$ $=(w_1\bm{a_1}, \ldots ,w_n\bm{a_n}, z_1\bm{a_1}, \ldots ,z_n\bm{a_n})$ is not surjective for $(\bm{z}, \bm{w}) \in ({\mu}^{-1}(\xi))_{\Sing}$. However, by the surjectivity of $A$, there exists $i$ such that $z_i=w_i=0$. Then,   
\[({\mu}^{-1}(\xi))_{\Sing}\subseteq\bigcup_{i=1}^n\left(({\mu}^{-1}(\xi))\cap\{z_i=w_i=0\}\right).\]
From the lemma above, we have $\dim \mu^{-1}(\xi)=2n-d$. Thus, it is enough to show that for each $i$, $\dim ({\mu}^{-1}(\xi)\cap\{z_i=w_i=0\})\leq2n-2-d$. To prove this, we consider the $d\times(n-1)$-matrix ${A^{(i)}}$ which is obtained by deleting the $i$-th column vector from $A$, and we note the following easy claim.  
\[A^{(i)}:\Z^{n-1}\to \Z^d \ \text{is surjective if and only if }\ \bm{b_i}\neq0.\]
Since the assumption implies that $A^{(i)}$ is surjective, we can consider the moment map $\mu^{(i)}:\C^{2(n-1)}\to \C^d$ corresponding to $A^{(i)}$. Also, it can be easily showed that $\mu^{(i)}$ is given by $\mu^{(i)}(z_1, \ldots, \hat{z_i}, \ldots ,z_n, w_1, \ldots ,\hat{w_i}, \ldots ,w_n)=\sum_{j\neq i}{z_jw_j\bm{a_j}}$ through the identification $\C^{2(n-1)}=\C^{2n}\cap\{z_i=w_i=0\}$. In particular, we have $(\mu^{(i)})^{-1}(\xi)=\mu^{-1}(\xi)\cap\{z_i=w_i=0\}$. Since we have $\dim (\mu^{(i)})^{-1}(\xi)=2n-2-d$ by the lemma above, we have $\codim_{{\mu}^{-1}(\xi)}({\mu}^{-1}(\xi))_{\Sing}\geq2$.  

Next, we want to show that $\mu^{-1}(\xi)/\hspace{-3pt}/\T_\C^d$ is normal and irreducible. By Lemma \ref{lem:CI}, ${\mu}^{-1}(\xi)$ is a complete intersection, so this is Cohen--Macaulay. Since we showed that $\codim_{{\mu}^{-1}(\xi)}({\mu}^{-1}(\xi))_{\Sing}\geq2$, we can show that ${\mu}^{-1}(\xi)$ is normal by Serre's normality criterion (\cite[Theorem 2.2.22]{BH}). 
On the other hand, by applying Knop's Theorem (\cite{Knop}) to this case, each fiber $\mu^{-1}(\xi)$ is connected. 
Thus, $\mu^{-1}(\xi)$ is irreducible. Moreover, in general, the invariant ring of an integrally closed domain by algebraic reductive group is integrally closed. Hence, $Y(A, (0, \xi))=\Spec \C[{\mu}^{-1}(\xi)]^{\T_\C^d}$ is a normal irreducible algebraic variety. 

Finally, we will show that $({\mu}^{-1}(\xi))^{0-st}$ is not empty, and $\pi$ is birational. Note that for any $k$, $z_kw_k$ is not in the defining ideal $I_\xi=\langle\{\xi_j-\sum_{i}{z_iw_ia_{ji}}\}\rangle_{j=1, \ldots ,n}\subset\C[\bm{z}, \bm{w}]$ of $\mu^{-1}(\xi)$. In fact, if $z_kw_k\in I_\xi$, then $z_k\in I_\xi$ or $w_k \in I_\xi$, since $I_\xi$ is prime ideal. However, one can easily show that $z_k\notin I_\xi$ and $w_k\notin I_\xi$ by the surjectivity of $A$ (more precisely, by the injectivity of $A^T$). 
{In fact, first, we can take a $\bm{v}=(v_1, \ldots, v_n)^T$ such that $(\xi_1, \ldots, \xi_d)=A\bm{v}$ by the surjectivity of $A$. If $z_k\in{I_\xi}$, then there exist polynomials $p_j(\bm{z}, \bm{w})$ satisfying}
\[z_k=\sum_{j=1}^d{p_j\left(\sum_{i=1}^n{(v_i-z_iw_i)a_{ji}} \right)}.\]
{However, this is a contradiction since if we substitute $\fr<v_i/z_i>$ for $w_i$ formally, then the right hand side is 0.} 


Thus, for each $k$, $z_kw_k=0$ defines an effective divisor $D_k$ in $\mu^{-1}(\xi)$. Now, we consider the $\T_\C^d$-invariant open subset $U:=(\mu^{-1}(\xi)\setminus\bigcup_{k}{D_k})\cap\mu^{-1}(\xi)^{\alpha-st}$ of $\mu^{-1}(\xi)$, where $\alpha$ is a generic element. 
Now, we want to show $U\subseteq\mu^{-1}(\xi)^{0-st}$. Since the stabilizer group of $\T_\C^d$ at any point in $U$ is finite, we only have to show that for any point $x$ in $U$, the $\T_\C^d$-orbit of $x$ is closed in $\mu^{-1}(\xi)=\mu^{-1}(\xi)^{0-ss}$. For any 1-parameter subgroup $\vphi_{\bm{v}}:\Cstar\to\T_\C^d \ : \ t\mapsto (t^{v_1}, \ldots ,t^{v_d})$, this acts on $\C^{2n}$ as $z_i\mapsto t^{u_i}z_i$, $w_i\mapsto t^{-u_i}w_i$, where $(u_1, \ldots ,u_n)=A^T\bm{v}$ and $A^T : (\Z^d)^* \to (\Z^n)^*$. By the injectivity of $A^T$, there exists $i$ such that $u_i\neq0$. By the definition of $U$, $z_i(x)\neq0$ and $w_i(x)\neq0$, so $\lim_{t\to\infty}\vphi_{\bm{v}}(t)\cdot x$ doesn't exist. By the Hilbert--Mumford's criterion, this implies that the orbit $\T_\C^d\cdot x$ is closed in $\mu^{-1}(\xi)$, so $U\subseteq\mu^{-1}(\xi)^{0-st}$. Then, this proves $(({\mu}^{-1}(\xi))^{0-st}\cap ({\mu}^{-1}(\xi))^{\alpha-st}){\supseteq} U\neq\emptyset$,  in particular, $({\mu}^{-1}(\xi))^{0-st}$ is not empty. 
Then, $\pi: Y(A, (\alpha, \xi))=({\mu}^{-1}(\xi))^{\alpha-st}/\T_\C^d \to Y(A, (0, \xi))=({\mu}^{-1}(\xi))^{0-st}/\hspace{-3pt}/\T_\C^d$ is the identity map on $(({\mu}^{-1}(\xi))^{0-st}\cap ({\mu}^{-1}(\xi))^{\alpha-st})/\T_\C^d$. In particular, $\pi$ is a birational morphism.   
\end{proof}

\begin{remark}
Since $\C^{2n}$ is clearly Cohen--Macaulay, so by the same argument above, one can prove that $X(A, 0)$ is a normal irreducible Poisson variety. Also, $(\C^{2n})^{0-st}\neq \phi$, in particular, $\dim X(A, 0)=2n-d$. Moreover, if $\alpha$ is generic, then $\Pi:X(A, \alpha) \to X(A, 0)$ can be a $\Cstar$-equivariant birational morphism. 
\end{remark}

\begin{proposition}\label{prop:resolution}
Let $({Y}, \om_0)$ be a nonsingular symplectic variety, and $Y_0$ be a normal algebraic variety. If a morphism $\pi:{Y} \to Y_0$ is a projective birational morphism, then $\pi$ will be a symplectic resolution. In particular, $Y_0$ is a symplectic variety.  
\qed\end{proposition}
\begin{proof}

We consider the open subset $U_0:=\{y \in {Y_0} \ | \ \dim \pi^{-1}(y) =0\}$. Then, by the Zariski's Main Theorem, {$\codim_Y (Y-U_0)\geq 2$} and the restriction $\pi|_{\pi^{-1}(U_0)}: \pi^{-1}(U_0) \xrightarrow{\sim} U_0$ is an isomorphism. 
We have a symplectic form $\om_0$ on $Y$, so the symplectic form $\overline{\om}_0$ on $U_0$ is induced through the isomorphism $\pi|_{\pi^{-1}(U_0)}$. Since $\Om_{(Y_0)_{\text{sm}}}^2$ is a locally free sheaf, $\overline{\om}_0$ is uniquely extended to a closed 2-form on $(Y_0)_{\text{sm}}$ (we also denote this by $\overline{\om}_0$). Moreover, this 2-form is also a symplectic 2-form. In fact, $\wedge^{\dim Y_0}\overline{\om}_0$ defines a global section of $K_{(Y_0)_{\text{sm}}}$, so if $\overline{\om}_0$ is degenerate, then the vanishing locus of $\wedge^{\dim Y_0}\overline{\om}_0$ defines a divisor $D$ in $(Y_0)_{\text{sm}}$. However, we have {$\codim_{(Y_0)_{\text{sm}}}{((Y_0)_{\text{sm}}-U_0)}\geq2$}, so $D\cap U_0\neq\phi$. This contradicts to the nondegeneracy of $\overline{\om}_0$ on $U_0$. Thus $\overline{\om}_0$ is a symplectic form. 

Next, we consider the pull-back $\pi^*\overline{\om}_0$ of $\overline{\om}_0$ by $\pi$, which is a 2-form on $\pi^{-1}((Y_0)_{\text{sm}})$. Since this coincides with $\om_0|_{\pi^{-1}((Y_0)_{\text{sm}})}$ on the open dense subset $\pi^{-1}(U_0)$, they define the same 2-form on $\pi^{-1}((Y_0)_{\text{sm}})$, that is, $(\pi|)^*\overline{\om}_0=\om_0$, where $\pi|: \pi^{-1}((Y_0)_{\text{sm}}) \to (Y_0)_{\text{sm}}$ is a restriction of $\pi$. 
Then, by the nondegeneracy of symplectic form, we have $(\pi|)^*K_{(Y_0)_{\text{sm}}}=K_{\pi^{-1}((Y_0)_{\text{sm}})}$, i.e., $\pi|$ is crepant. 
Since $\pi|: \pi^{-1}((Y_0)_{\text{sm}}) \to (Y_0)_{\text{sm}}$ is a crepant projective birational morphism between nonsingular algebraic varieties, $\pi|$ is {an isomorphism}.   
\end{proof}


Then, we get the following maybe well-known theorem on hypertoric varieties. 
\begin{theorem}\label{thm:hyperfund}
For a unimodular $d\times n$-matrix $A$ and a generic $\alpha$, the following hold. 
\begin{itemize}
\item [(1)]$X(A, \alpha)$ is a $2n-d$ dimensional smooth Poisson variety, and $Y(A, \alpha)$ is a $2n-2d$ dimensional smooth symplectic variety.  
\item [(2)]For any $\xi\in\C^d$, $\pi_\xi: Y(A, (\alpha, \xi))\to Y(A, (0, \xi))$ is a projective symplectic resolution. Moreover, $\pi : Y(A, \alpha)\to Y(A, 0)$ is a conical projective symplectic resolution with respect to the conical $\Cstar$-action which is defined in Remark \ref{rem:comm diagram}.  
\end{itemize}
\end{theorem}
\begin{proof}
(1) is already proved in Corollary \ref{cor:hypertoricsymplectic}. 

(2) The former part is deduced from Proposition \ref{prop:birational} and Proposition \ref{prop:resolution}. Then, we only have to show that $Y(A, 0)$ is a conical symplectic variety. The symplectic 2-form $\overline{\om}_0$ on $(Y(A, 0))_{\text{sm}}$ is naturally induced from the standard symplectic 2-form $\om_\C$ on $\C^{2n}$ (see Remark \ref{rem:comm diagram}). By definition, we have $s^*\om_\C=s^{2}\om_\C$ ($s\in\Cstar$), so $s^*\overline{\om}_0=s^2\overline{\om}_0$. Thus, $Y(A, 0)$ is a conical symplectic variety. 
\end{proof}

For the later convenience, we compute the coordinate ring of $X(A, 0)$ and $Y(A, 0)$. First, note the following exact sequence obtained from the original exact sequence. 
\[
\begin{tikzcd}
0\ar[r]&\Z^{2n-d}\ar[r, "\widehat{B}"] &\Z^{2n}\ar[r, "\widehat{A}"]&\Z^d\ar[r]&0         
\end{tikzcd},
\]
where 
\[\widehat{A}:= 
\begin{pmatrix}
&&&\\
\vcbig{A}&\vcbig{-A}\\
\end{pmatrix}, \ \ \ 
\widehat{B}:=
\left(\begin{array}{ll|rr}
&&&\\
\vbig{B}&&&\vbig{I_n}\\
\cline{1-4}
&&&\\
\vbig{0}&&&\vbig{I_n}
\end{array}\right).
\]
($\widehat{B}$ is called the {\it Lawrence lift} of $B$.)
\begin{lemma}{\label{ring str lem}}
\begin{align*}
\C[X(A, 0)]&=\C[z_1w_1,\ldots ,z_nw_n][f_{\bm{\beta}} \ | \ \bm{\beta}\in \Image B],
\end{align*}
\[\C[Y(A, 0)]=\C[z_1w_1,\ldots ,z_nw_n][f_{\bm{\beta}} \ | \ \bm{\beta}\in \Image B]/\langle\sum_{j=1}^n{z_jw_ja_{ij}} | i=1, \ldots ,d\rangle,
\] 
where we define $f_{\bm{\beta}}:=\prod_{i: \beta_i>0}{z_i^{\beta_i}}\prod_{i: \beta_i<0}{w_i}^{-\beta_i}$. 
\end{lemma}
\begin{proof}
By definition, 
\begin{align*}
\C[X(A, 0)]&=\C[\bm{z}, \bm{w}]^{\T_\C^d}=\bigoplus_{\text{{\scriptsize$\begin{array}{l}\bm{u}, \bm{v} \in \Z_{\geq 0}\\ s.t. \ A\bm{u}-A\bm{v}=0\end{array}$}}}{\C{\bm{z}}^{\bm{u}}{\bm{w}}^{\bm{v}}}=\bigoplus_{{}^t\!(\bm{u}, \bm{v}) \in \Image(\widehat{B}) \cap \Z_{\geq 0}^{2n}}{\C{\bm{z}}^{\bm{u}}{\bm{w}}^{\bm{v}}} \\
&=\bigoplus_{\text{{\scriptsize$\begin{array}{l}\bm{\beta} \in \Image(B), \bm{\ga} \in \Z_{\geq 0}^n\\ s.t. \ \bm{\beta} + \bm{\ga} \geq 0\end{array}$}}}{\C\bm{z}^{\bm{\beta}}(z_1w_1)^{\ga_1} \cdots (z_nw_n)^{\ga_n}}, 
\end{align*}
where $\bm{\beta}+\bm{\gamma}\geq0$ means that $\beta_i+\gamma_i\geq0 $ for all $i$.  
It can be easily checked that $\C[X(A, 0)] \supseteq \C[z_1w_1,\ldots ,z_nw_n][f_{\bm{\beta}} \ | \ \bm{\beta}\in \Image B]$. On the other hand, for $\bm{\beta}\in\Image B$ and $\bm{\ga}\in \Z_{\geq 0}^n$ such that $\bm{\beta} + \bm{\ga} \geq 0$, we will show $\bm{z}^{\bm{\beta}}(z_1w_1)^{\ga_1} \cdots (z_nw_n)^{\ga_n}\in\C[z_1w_1,\ldots ,z_nw_n][f_{\bm{\beta}} \ | \ \bm{\beta}\in \Image B]$. 

Dividing $\bm{z}^{\bm{\beta}}(z_1w_1)^{\ga_1} \cdots (z_nw_n)^{\ga_n}=\displaystyle\prod_{i=1}^n{z_i^{\beta_i + \ga_i}}\displaystyle\prod_{i=1}^n{w_i}^{\ga_i}$ by $z_jw_j \ (1\leq j\leq n)$ as many as possible, we can assume for each $j$, $\beta_j+\ga_j=0$ or $\ga_j=0$. Then
\[\displaystyle\prod_{i=1}^n{z_i^{\beta_i + \ga_i}}\displaystyle\prod_{i=1}^n{w_i}^{\ga_i}=\prod_{i: \beta_i>0}{z_i^{\beta_i}}\prod_{i: \beta_i<0}{w_i}^{-\beta_i} = f_{\bm{\beta}}.\] 
This completes the proof.
\end{proof}
\begin{remark}\hspace{2pt}\label{rem:deg=1}\\
\begin{itemize}
\vspace{-5mm}
\item[(1)]By this lemma, if $\bm{z}^{\bm{u_1}}\bm{w}^{\bm{u_2}} \in \C[X(A, 0)]$, then $\bm{z}^{\bm{u_2}}\bm{w}^{\bm{u_1}} \in \C[X(A, 0)]$. In addition, by symmetry, if $\bm{z}^{\bm{u_1}}\bm{w}^{\bm{u_2}}$ is one of the minimal generators of $\C[X(A, 0)]$, then $\bm{z}^{\bm{u_2}}\bm{w}^{\bm{u_1}}$ is also one of the minimal generators.    
\item[(2)]If for some $\ell>0$ and $1\leq i\leq n$, $z_i^{\ell}$ (resp. $w_i^{\ell}$) $\in\C[X(A, 0)]$, then by the lemma above, we will have $\ell\bm{e_i}\in\Image B=\Ker A$. This means $\bm{a_i}=0$. 
Thus if for any $i$, $\bm{a_i}\neq0$, then $z_i^{\ell}, w_i^{\ell}\notin\C[X(A, 0)]$. In particular, $\C[X(A, 0)]$ is generated by monomials of degree greater than 1, and by (1) the minimal generators can be written in the following form: 
\[\{z_1w_1, \ldots ,z_nw_n\}\bigsqcup\{f_{\bm{\beta_k}}, f_{-\bm{\beta_k}} \ | \ k=1, \ldots ,m\},\]
where $\bm{\beta_k}$ is some element of $\Image B$. 
\item[(3)]Since the Poisson structure on $\C^{2n}$ is given by $\{f, g\}=\displaystyle\sum_{j=1}^n{\fr<\partial f/\partial z_j>\fr<\partial g/\partial w_j>-\fr<\partial f/\partial w_j>\fr<\partial g/\partial z_j>}$ ($f, g\in\C[\bm{z}, \bm{w}]$), one can compute explicitly the induced Poisson structure on $X(A, 0)$ and $Y(A, 0)$ as the following: 
\[\{z_jw_j, z_kw_k\}=0,  \ \ \{f_{\bm{\beta}}, z_kw_k\}=\beta_kf_{\bm{\beta}},  \ \ \ \{f_{\bm{\beta}}, f_{\bm{\gamma}}\}=\sum_{j: \beta_j\gamma_j<0}{\beta_j|\gamma_j|\fr<f_{\bm{\beta}}f_{\bm{\gamma}}/z_jw_j>},\]
where $|\gamma_j|$ is the absolute value of $\gamma_j$. 
\end{itemize}
\end{remark}

We give some typical examples of hypertoric varieties. 

\begin{example}\label{Amsurf}(The $A_m$ type surface singularity and its minimal resolution ($m\geq1$))\\
Consider the following exact sequence. 
\[
\begin{tikzcd}[row sep=large, column sep=1ex, ampersand replacement=\&]
0\ar[r]\&[5ex] \Z \ar[r, "{B=\begin{pmatrix} 1\\1\\\vdots \\1 \end{pmatrix}}"] \&[25ex]\Z^{m+1} \ar[r, "{A=\begin{pmatrix}1&&     && &-1\\&\ddots&&\vcbig{0}  &\vdots \\&  & \ddots&&  &\vdots \\ \vcbig{0}&&&1&-1\end{pmatrix}}"] \&[25ex]\Z^m\ar[r]\&[5ex] 0
\end{tikzcd}
\]
For a generic $\alpha$, the corresponding symplectic resolution $\pi : Y(A, \alpha) \to Y(A, 0)$ is given by the following (cf.\ the proof of Theorem \ref{thm:4dimension}).  
\[
\begin{tikzcd}
Y(A, \alpha)\ar[r, "\sim"]\ar[d, "\pi"]&\tilde{S}_{A_m}\ar[d]&&:\text{the minimal resolution}\\
Y(A, 0)\ar[r, "\sim"]&S_{A_m}:\ar[r, equal]&\left\{\det\begin{pmatrix}u_1&x_1\\y_1&u_1^m\end{pmatrix}=0\right\}&:\text{the $A_m$ type surface singularity}
\end{tikzcd}
\]
\end{example}

\begin{example}{\label{Amnilpotent}}(The $A_m$ type minimal nilpotent orbit closure ($m\geq1$))\\
Consider the following exact sequence. 
\[
\begin{tikzcd}[row sep=large, column sep=1ex, ampersand replacement=\&]
0\ar[r]\&[5ex]\Z^m \ar[r, "{B=\begin{pmatrix} 1 & &&\\ &\ddots&\vcbig{0}\\  &  &\ddots&\\ \vcbig{0}   &   &1\\ -1&\cdots&\cdots&-1\end{pmatrix}}"] \&[25ex] \Z^{m+1} \ar[r, "{A=\begin{pmatrix}1&1&\cdots&1\end{pmatrix}}"] \&[25ex]\Z\ar[r]\&[5ex] 0
\end{tikzcd}
\]
For a generic $\alpha>0$, one can prove directly that the corresponding symplectic resolution $\pi : Y(A, \alpha) \to Y(A, 0)$ is given by the following.
\[
\begin{tikzcd}[column sep=tiny, ampersand replacement=\&]
Y(A, \alpha)\ar[r, equal]\ar[d, "\pi"]\&\left\{(\bm{z}, \bm{w}) \in \C^{2m+2} \ \left|\right. \ \bm{z}\neq0, \ \sum_{j=1}^{m+1}{z_jw_j}=0\right\}/\T_\C^1\ar[r, "\sim"]\ar[d, "\pi"]\&T^*\bm{P}^m\ar[d, "\pi'"]\\
Y(A, 0)\ar[r, equal]\&\Set{C\in \gersl_{m+1} \ | \ \text{All $2\times2$-minors of $C=0$}}\ar[r, equal]\&\overline{\calO}_{A_m}^{\tmin}
\end{tikzcd},
\]
where $\pi'$ is the Springer resolution. Concretely, $\pi$ is given by the following. 
\[
\pi(\bm{z}, \bm{w}):=\begin{pmatrix}z_1w_1&z_1w_2&\cdots&z_1w_{m+1}\\ w_1z_2&z_2w_2&\ddots&\vdots\\ \vdots&\ddots&\ddots&z_{m}w_{m+1}\\ w_1z_{m+1}&\cdots&w_{m}z_{m+1}&z_{m+1}w_{m+1} \end{pmatrix}.
\]

\end{example}
\begin{example}(Toric quiver varieties cf.\ \cite[section 8]{HS})\label{eg:quiver}\\
Let $G=(\{v_0, \ldots ,v_d\}, E)$ be a connected undirected finite graph with $d+1$ vertices $\{v_0, \ldots ,v_d\}$ and $n$ edges $\{e_{ij}^k:=\{i, j\} \ | \ k=1, \ldots, m_{ij}\}$, where $m_{ij}$ is the multiplicity of each edge $\{i, j\}$. We also assume that $G$ doesn't have any self-loops. Next, we assign an orientation $\vec{e}^{\,k}_{ij}:=i\to j$ or $j\to i$ to each edge $e_{ij}^k$ such that $v_0$ is a {\it source}, that is, there doesn't exist $j$ and $k$ such that $e_{j0}^k\in E$ (we can always assume this by Remark \ref{rem:quiver1} (2) below).
In this setting, one can consider the natural surjective map 
\[A_G : \bigoplus_{e_{ij}^k\in E}\hspace{-5pt}\Z e_{ij}^k \cong\Z^{n} \twoheadrightarrow N\subset\bigoplus_{k=0}^d{\Z v_k}\cong\Z^{d+1} : e_{ij}^k \to\begin{cases}v_i-v_j&(\text{if} \ \vec{e}^{\,k}_{ij}=j\to i)\\v_j-v_i &(\text{if} \ \vec{e}^{\,k}_{ij}=i\to j)\end{cases} ,\]
where $N:=\{\sum_{k=0}^d{\gamma_kv_k } \ | \ \sum_{k=0}^d{\gamma_k}=0\}\cong\Z^{d}$. 
If we take a basis of $N\cong\Z^d$ as $v_0-v_1, \ldots ,v_0-v_d$, and we consider $A_G$ as a $d\times n$-matrix, then the induced action $\T_N=\T_\C^d\curvearrowright\C^{2|E|}$ will be described as the following: 
\begin{align*}\bm{t}\cdot&(\ldots, z_{0j}, \ldots, z_{ij}, \ldots, \ldots, w_{0j}, \ldots, \ldots, w_{ij}, \ldots)\\
&=(\ldots, t_jz_{0j}, \ldots, t_jz_{ij}t_i^{-1}, \ldots, \ldots, w_{0j}t_j^{-1}, \ldots, \ldots, t_iw_{ij}t_j^{-1}, \ldots).\end{align*}
Then, we call $Y(A_G, \alpha)$ a {\it toric quiver variety}. 
As below, toric quiver varieties can be also described as Nakajima quiver varieties. Actually, for a graph $G$, we can associate the double quiver $\overline{Q}_G$ to it by ``expanding'' $G$ at $v_0$ ($v_0$ becomes $q$ white dots, where $q$ is the number of edges going out from $v_0$) and adding both orientations to all edges. For example, 
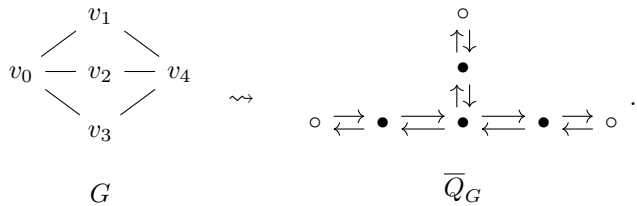
\[\begin{tikzcd}[column sep=small, row sep=small]
&v_1\ar[dr, no head]&\\
v_0\ar[r, no head]\ar[dr, no head]\ar[ur, no head]&v_2\ar[r, no head]&v_4\\
&v_3\ar[ur, no head]&\\
&G&
\end{tikzcd}
 \\ \ \  \rightsquigarrow \ \ \ \ 
\begin{tikzcd}[column sep=small, row sep=small]
&&\circ\ar[d, shift left]&&\\
&&\bullet\ar[u, shift left]\ar[d, shift left]&&\\
\circ\ar[r, shift left]&\bullet\ar[r, shift left]\ar[l, shift left]&\bullet\ar[u, shift left]\ar[l, shift left]\ar[r, shift left]&\bullet\ar[l, shift left]\ar[r, shift left]&\circ\ar[l, shift left]\\
&&\overline{Q}_G&&
\end{tikzcd}.\]
Then, by the definition of the torus action, $Y(A_G, \alpha)$ is the same as the (framed) Nakajima quiver variety $\gerM_{\alpha}(\overline{Q}_G, \bm{v}, \bm{w})$, where its dimension vector is $\bm{v}=(1, \ldots ,1)\in\Z^{p}$ and $\bm{w}=(1, \ldots ,1)\in\Z^{q}$, where $p$ (resp. $q$) is the number of $\bullet$ (resp. $\circ$) of $\overline{Q}_G$.

Toric quiver varieties include many examples of hypertoric varieties, actually, the hypertoric variety in Example \ref{Amsurf} is the toric quiver variety associated to the edge graph of $(m+1)$-gon and the one in Example \ref{Amnilpotent} is the toric quiver variety associated to the following graph with $m+1$ {multiple edges}: 
\begin{figure}[h]
\centering
\begin{tikzcd}\bullet\arrow[r, draw=none, "\raisebox{+1.5ex}{\vdots}" description, no head]\arrow[r, bend left,"e^0", no head]\arrow[r, bend right, swap, "e^{m}", no head]&\bullet\end{tikzcd}
\caption{Example \ref{Amnilpotent}}
\end{figure}
\end{example}

\begin{remark}\label{rem:quiver1}\hspace{2pt}\\
\begin{itemize}
\vspace{-5mm}
\item[(1)]If $G$ has some connected components $G=G_1\sqcup\cdots\sqcup G_s$, then we can naturally define the corresponding hypertoric variety to $G$ is the direct product of the corresponding ones to $G_k$'s (cf.\ Lemma \ref{lem:product}). 
\item[(2)]Although we assign an orientation to $G$ such that $v_0$ is a source, the $\Cstar\times\T_\C^{n-d}$-equivariant isomorphism class of the associated hypertoric variety as Poisson variety doesn't depend on the orientation of $G$. Actually, since switching the orientation of each edge $e_{ij}^k$ is the same as  multiplying the column vector of $A_G$ corresponding to $e_{ij}^k$ by $-1$, as we will see later (cf.\ Lemma \ref{lem:fund}), this operation induces an $\Cstar\times\T_\C^{n-d}$-equivariant isomorphism between the corresponding hypertoric varieties.   
\item[(3)]In the section 4, we will classify 4- and 6-dimensional affine hypertoric varieties, and then in these cases, all hypertoric varieties are realized as toric quiver varieties.
\end{itemize}
\end{remark}

In the remaining part of this section, we consider some transformations of $A$ (or $B$) which don't change the {$\Cstar\times\T_\C^{2n-d}$-equivariant (resp. $\Cstar\times\T_\C^{n-d}$-equivariant)} isomorphism class of $X(A, \alpha)$ (resp. $Y(A, \alpha)$) as Poisson varieties, {where $\T_\C^{2n-d}:=\T_\C^{2n}/\T_\C^d$ and $\T_\C^{n-d}:=\T_\C^n/\T_\C^d$ is the remaining hamiltonian torus action.} After that, we explain that two (unimodular) matrices $A$ and $A'$ are transformed by those transformations if and only if the associated matroids $M(A)$ and $M(A')$ are isomorphic to each other. 

\begin{definition}\label{def:fund}
For rank $d$ matrices $A$, $A'\in\Mat_{d\times n}(\Z)$, we say that $A'$ is {\it equivalent} to $A$ if $A'$ is obtained from $A$ by a sequence of some elementary row operations over $\Z$, interchanging some column vectors $\bm{a_i}$, and multiplying some $\bm{a_i}$ by $-1$. In this case, we denote $A\sim A'$.  
\end{definition}
\begin{remark}\label{rem:fund}
By definition, $A\sim A'$ holds if and only if there exists $P\in GL_d(\Z)$ and a $n\times n$ signed permutation matrix $D$ such that $A'=PAD$. Here, a signed permutation matrix is a product of a permutation matrix and a diagonal matrix whose all diagonal components are 1 or $-1$.  
\end{remark}
Then, we have the following (this is essentially considered in \cite[Proposition 3.2]{AP}). 
\begin{lemma}\label{lem:fund} In the setting above, if $A\sim A'$ (more explicitly $A'=PAD$ as in Remark \ref{rem:fund}), then we have the following $\Cstar\times\T_\C^{2n-d}$-equivariant {(resp. $\Cstar\times\T_\C^{n-d}$-equivariant)} isomorphisms as Poisson varieties {(resp. symplectic varieties)}.  
\[ (X(A, \alpha), \{-, -\}) \xrightarrow{\sim} (X(A', \alpha'), \{-, -\}'),
\]
\[ (\text{resp. } \ (Y(A, \alpha), \om_0) \xrightarrow{\sim} (Y(A', \alpha'), \om'_0)),
\]
where $\alpha':=P\alpha$. 
\end{lemma}
\begin{proof}
Since elementary row operations don't change the image of torus embedding $A^T: \T_\C^d\hookrightarrow \T_\C^n$, the claim is clear if we take $\alpha':=P\alpha$. 

Next, we denote the matrix corresponding to interchanging the $i$-th column vector $\bm{a_i}$ and the $j$-th column vector $\bm{a_j}$ of $A$ by $D_{ij}$ and set $A'=AD_{ij}$. 
Then, the moment map $\mu': \C^{2n}\to \C^d$ corresponding to $A'$ is the same as the composition $\mu\circ \widehat{D_{ij}}$, where $\widehat{D_{ij}}: \C^{2n}\to\C^{2n}$ is the $\Cstar\times\T_\C^{2n}$-equivariant linear isomorphism as only interchanging $z_i\leftrightarrow z_j$ and $w_i\leftrightarrow w_j$. Since $\widehat{D_{ij}}$ preserves the symplectic form $\om_\C:=\sum_{j=1}^n{dz_j\wedge dw_j}$ clearly, this induces the desired $\Cstar\times\T_\C^{2n-d}${-equivariant isomorphism $X(A, \alpha) \xrightarrow{\sim} X(A', \alpha')$. Moreover, $\T_\C^n$-invariance of $\mu$ induces the desired $\Cstar\times\T_\C^{n-d}$-equivariant isomorphism $Y(A, \alpha)\xrightarrow{\sim} Y(A', \alpha')$. }

Finally, we set the matrix $D_i$ corresponding to multiplying the $i$-th column vector $\bm{a_i}$ by $-1$, and we assume $A'=AD_i$. Then, the moment map $\mu': \C^{2n}\to \C^d$ corresponding to $A'$ is the same as the composition $\mu\circ \widehat{D_{i}}$, where $\widehat{D_{i}}: \C^{2n}\to\C^{2n}$ is the $\T_\C^n$-equivariant linear isomorphism {sending as} $(z_i, w_i)\mapsto (-w_i, z_i)$. Since $\widehat{D_{i}}$ preserves the symplectic form $\om_\C:=\sum_{j=1}^n{dz_j\wedge dw_j}$, this induces the desired {$\Cstar\times\T_\C^{2n-d}$-equivariant isomorphism $X(A, \alpha) \xrightarrow{\sim} X(A', \alpha')$ (resp. $\Cstar\times\T_\C^{n-d}$-equivariant isomorphism $Y(A, \alpha)\xrightarrow{\sim} Y(A', \alpha')$)}. 
\end{proof}

\begin{remark}\hspace{2pt}\label{rem:fund}\\
\begin{itemize}
\vspace{-5mm}
\item[(1)]We {took $A$, then we take} $B$ as $\begin{tikzcd}[column sep=small]0\ar[r]&\Z^{n-d}\ar[r, "{B}"] &\Z^{n}\ar[r, "{A}"]&\Z^d\ar[r]&0\end{tikzcd}$ is exact. Conversely, if we take $B$ at first and take $A'$ as the sequence above is exact, then by the lemma above, the {$\Cstar\times\T_\C^{2n-d}$-equivariant (resp. $\Cstar\times\T_\C^{n-d}$-equivariant)} isomorphism class doesn't depend on the choice of $A'$. Thus, we can also think that $X(A, \alpha)$ (resp. $Y(A, \alpha)$) is defined by $B$ and $\alpha$. In similar way, if we consider a matrix $B'$, which is obtained from $B$ by some elementary column operations over $\Z$, interchanging row vectors $\bm{b_i}$, and multiplying $\bm{b_i}$ by $-1$, then the {$\Cstar\times\T_\C^{2n-d}$-equivariant (resp. $\Cstar\times\T_\C^{n-d}$-equivariant)} isomorphism class of the corresponding Lawrence toric variety (resp. hypertoric variety) would not change.

In this paper, we fix a pair $(A, B)$ and consider the corresponding Lawrence toric variety (resp. hypertoric variety). Hence, if we simply say that we take $A'=PAD$ as Remark \ref{rem:fund}, then we implicitly take a $B'=DBQ$ and consider a pair $(A', B')$ simultaneously, where some $Q\in GL_{n-d}(\Z)$. 
\item[(2)]
Let $A\in \Mat_{d\times n}(\Z)$ and $B\in \Mat_{n\times n-d}(\Z)$ be unimodular matrices. By applying some transformations to $A$ and $B$, one can transform them to matrices in the following forms: 
\[B'=\left(
\begin{array}{cccc}
1&&&\\
&\ddots&\vcbig{0}\\
&&\ddots&\\
\vcbig{0}&&1\\
\cline{1-4}
&&&\\
&\vcbig{C}&\\
\end{array}
\right)\\
, \ \ \ A'=
\left(
\begin{array}{cccc|ccc}
&&&&1&          &\LargeO\\
&\vcbig{-C}&&  &\ddots&\\
&&&&\LargeO  &          &1\\
\end{array}
\right),
\]
where $C\in \Mat_{d\times n-d}(\Z)$ is a totally unimodular matrix. Here, we call $C$  {\it totally unimodular} matrix if every square submatrix has determinant $0$ or $\pm1$. In this case, as one can easily see, $A'$ and $B'$ are also totally unimodular matrices.  
\end{itemize}
\end{remark}

The following is obvious. 
\begin{lemma}\label{lem:product}If $A\in\Mat_{d\times n}(\Z)$ is the form of $A=A_1\oplus A_2\oplus\cdots\oplus A_s$, where
$A_i\in M_{k_i\times \ell_i}(\Z)$, then there will exist the following natural $\Cstar\times\T_\C^{n-d}$-equivariant isomorphism as Poisson varieties. 
\[X(A, \alpha) \cong X(A_1, \alpha_1) \times X(A_2, \alpha_2) \times \cdots \times X(A_s, \alpha_s),
\]
\[Y(A, \alpha) \cong Y(A_1, \alpha_1) \times Y(A_2, \alpha_2) \times \cdots \times Y(A_s, \alpha_s),
\]
where $\alpha_i:=p_{i}(\alpha)$, and $p_i: \C^d \to \C^{k_i}$ is the natural projection. 
\end{lemma}
\begin{remark}\label{rem:product}
In this case, we can take $B$ as $B_1\oplus\cdots\oplus B_s$, where each $B_i$ is a $\ell_i\times(\ell_i-k_i)$-matrix satisfying the exact sequence $\begin{tikzcd}[column sep=tiny]0\ar[r]&\Z^{\ell_i-k_i}\ar[r, "{B_i}"] &\Z^{\ell_i}\ar[r, "{A_i}"]&\Z^{k_i}\ar[r]&0\end{tikzcd}$. 
\end{remark}

\begin{example}{\label{eg:product}}(The direct product of some minimal nilpotent orbit closures of type $A$)\\
For $A:=
  \begin{pmatrix}
  \smash[b]{\block{\ell_1+1}}&                              &          &                     \\
                              & \smash[b]{\block{\ell_2+1}} &          &  \text{\huge{0}}\\
                              &                              & \ddots &                     \\
   \text{\huge{0}}      &                              &           & \block{\ell_s+1}
  \end{pmatrix}
$
and a generic $\alpha \in \Z^s$, we obtain the following isomorphism. 
\[
\begin{tikzcd}
Y(A, \alpha)\ar[r, "\sim"]\ar[d]&T^*\bm{P}^{\ell_1} \times T^*\bm{P}^{\ell_2} \times \cdots \times T^*\bm{P}^{\ell_s}\ar[d]\\
Y(A, 0)\ar[r, "\sim"]&\overline{\calO}_{A_{\ell_1}}^{\tmin} \times \overline{\calO}_{A_{\ell_2}}^{\tmin} \times \cdots \times \overline{\calO}_{A_{\ell_s}}^{\tmin}
\end{tikzcd},
\]
where if $\ell=0$, then $\overline{\calO}_{A_{\ell}}^{\tmin}:=\{pt\}$.
In the section 4, we will show that all hypertoric varieties {that} are isomorphic to nilpotent orbit closures $\overline{\calO}$ of some semisimple Lie algebra $\gerg$ as conical symplectic varieties, are in the form above (Theorem \ref{thm:char_nilpotent}). 
\end{example}


In the remaining part of this paper, we assume that for any $i$, $\bm{b_i}\neq0$. This is because if $\bm{b_i}=0$ for some $i$, then we can choose $A$ as the following form:  
\[
A=\left(\begin{array}{ccc|c|ccc}
0&\cdots&0&1&0&\cdots&0\\
\cline{1-7}
  &         & &0&  &         &\\
\vcbig{A_1} &&\vdots&\vcbig{A_2}         &\\
  &         & &0&          &         &
\end{array}\right).
\]
Then, by interchanging $\bm{a_1}$ with $\bm{a_i}$, we can apply Lemma \ref{lem:product} to this $A$ and obtain the following: 
\begin{corollary}\label{cor:Breduction}
Assume $\bm{b_i}=0$ for some $i$. For $A$ as above, we set $\overline{A}\in M_{(d-1)\times(n-1)}(\Z)$ as the following: 
\[
\overline{A}:=\begin{pmatrix}
&&&\\
\vcbig{A_1}&\vcbig{A_2}\\
\end{pmatrix}.
\]
Then, there exists the following $\Cstar\times\T_\C^{n-d}$-equivariant isomorphism as Poisson varieties
\[
X(A, \alpha) \cong X(\overline{A}, \overline{\alpha})\times\C,
\]
\[
Y(A, \alpha) \cong Y(\overline{A}, \overline{\alpha}),
\]
where $\overline{\alpha}$ is the image of $\alpha$ by the natural projection $\C^d \to \C^{d-1} : (u_1, \ldots ,u_d) \mapsto (u_2, \ldots ,u_d)$. 
\end{corollary}
\begin{proof}
By Lemma \ref{lem:product}, we only have to show that for $\begin{tikzcd}[column sep=1ex]0\ar[r]&[1ex]0\ar[r, "{B_0:=0}"] &[5ex]\Z\ar[r, "{A_0:=(1)}"]&[5ex]\Z\ar[r]&[1ex]0\end{tikzcd}$, we have $X(A_0, \alpha_1)\cong\C$ and $Y(A_0, \alpha_1)\cong\{pt\}$. This is easily followed by the definition of the moment map. 
\end{proof}
Since in the section 3 we will classify affine {hypertoric} varieties by the associated regular matroids, here we review some notions of matroids. For the more detail on matroids, we refer the reader to \cite{Ox}. 
For reader's convenience, we review the definition (actually there are many equivalent definitions of matroids) and some properties of matroids which we will need later.    
\begin{definition}\hspace{2pt}\label{def:matroid}\\
\begin{itemize}
\vspace{-5mm}
\item[(1)]A {\it matroid} $M$ is an ordered pair $(E, \calB)$ consisting of a finite set $E$ and a collection $\calB$ of subsets of $E$ satisfying the following properties: 
\begin{itemize}
\item[(i)]$\calB\neq\emptyset$.
\item[(ii)]If $B_1$ and $B_2$ are in $\calB$ and $x\in B_1\setminus B_2$, then there is an element $y\in B_2\setminus B_1$ such that $\{y\}\cup(B_1\setminus\{x\})\in\calB$. 
\end{itemize}
In this case, we call each $B\in\calB$ a {\it basis} of $M$. 
\item[(2)]For a matroid $M=(E, \calB)$, we define the {\it dual matroid} $M^*$ of $M$ as $M^*:=(E, B^*:=\{E\setminus B \ | \ B\in\calB\})$. 
By definition, we have $(M^*)^*=M$. 
\item[(3)]Let $M_1$ and $M_2$ be two matroids on the same ground set $E$. If a bijective map $\vphi: E\to E$ satisfies the following, we say $\vphi$ is a  {\it matroid isomorphism}: 
\[\text{$B\subset E$ is a basis of $M_1$ if and only if $\vphi(B)\subset E$ is a basis of $M_2$}.\]
\end{itemize}
\end{definition}
Then, for example, if an integer matrix $A=[\bm{a_1}, \ldots ,\bm{a_n}]$ is given, then one can consider {\it the vector matroid $M(A)=([n]:=\{1, \ldots ,n\}, \calB)$ associated to $A$ over $\Q$} as the following: 
\[\calB:=\{\{i_1, \ldots ,i_{d}\}\subseteq[n] \ | \ \{\bm{a_{i_1}}, \ldots ,\bm{a_{i_d}}\} \ \text{is a basis over $\Q$}\}.\]

\begin{remark}\label{rem:dual matroid}
For an exact sequence $\begin{tikzcd}[column sep=small]0\ar[r]&\Z^{n-d}\ar[r, "{B}"] &\Z^{n}\ar[r, "{A}"]&N\cong\Z^d\ar[r]&0\end{tikzcd}$, the dual matroid of $M(A)$ is given by $M(B^T)$. One can see this by using Lemma \ref{lem:Gale} for $r=0$ (or see also \cite[Proposition 2.2.23]{Ox}). 
\end{remark}

In Definition \ref{def:fund}, we considered some operations on matrices and an equivalence relation. Clearly, these transformations don't change the isomorphism class of the associated vector matroid as the following: 
\begin{lemma}
Let $A$ be a matrix. For a matrix $A'$ such that $A'\sim A$, we have
\[M(A)\cong M(A'). \] 
\end{lemma}
Below, if a matroid $M$ is isomorphic to a matroid $M(A)$ associated to some totally unimodular matrix $A$, we call $M$ {\it regular matroid}. By Remark \ref{rem:fund} (3), the associated matroid $M(A)$ to some (not necessary totally) unimodular matrix $A$ is regular. 
Then, the following important fact that we will need later is the converse of the lemma above in the case of unimodular matrices. 
\begin{theorem}
For two unimodular matrices $A$ and $A'$ whose size are the same, the following does hold: 
\[A\sim A' \ \Leftrightarrow M(A) \cong M(A').\]
\end{theorem}
\begin{proof}
We only have to show that if $M(A)\cong M(A')$, then $A\sim A'$. First, we note that we can replace $A$ (resp. $A'$) by a totally unimodular matrix respectively (cf.\ Remark \ref{rem:fund} (3)). Then, by \cite[Proposition 6.5]{Ox}, we know that $A\sim_{\Q}A'$. Here $A\sim_{\Q}A'$ means that there exists $P\in GL_d(\Q)$ and a matrix $D\in GL_n(\Q)$ such that $A'=PAD$, where $D$ is a product of a permutation matrix and a diagonal matrix. 
Its proof depends on Lemma 6.6.4 and Theorem 6.4.7 in \cite{Ox}. 
However, by reading these proofs in \cite{Ox} carefully, for totally unimodular matrices, it is easily showed that one can take  $P\in GL_d(\Z)$ and a signed permutation matrix $D\in GL_n(\Z)$ (cf.\ Remark \ref{rem:fund}) such that $A'=PAD$, i.e., $A\sim A'$. This completes the proof. 
\end{proof}
We will use this theorem to show that the isomorphism classes of affine hypertoric varieties $Y(A, 0)$ as conical symplectic varieties are bijective to the isomorphism classes of regular matroids (see Theorem \ref{thm:classification}). Finally, we consider some subclasses of regular matroids. 
\begin{definition}(Graphic matroids and cographic matroids)\\
Let $M$ be a matroid. We call $M$ a {\it graphic matroid} if $M$ is isomorphic to some matroid $M(G):=M(A_G)$, where $G$ is some finite  graph and $A_G$ is its signed incidence matrix with respect to some (or equivalently any)  orientation (cf.\ Example \ref{eg:quiver}). 

If a matroid $M$ is the dual matroid of some graphic matroid, we call $M$ {a} {\it cographic matroid}. 
\end{definition}
\begin{remark}\label{rem:graphic} \hspace{2pt}\\
\begin{itemize}
\vspace{-5mm}
\item[(1)]By the similar reason as we mentioned in Remark \ref{rem:quiver1} (2), the isomorphism class of $M(A_G)$ doesn't depend on the choice of the  orientation of $G$. 
\item[(2)]It is known that a matroid $M$ is graphic and cographic if and only if $M\cong M(G)$, where $G$ is some planar graph (cf.\ \cite[Theorem 5.2.2]{Ox}). In this case, the dual matroid $M(G)^*$ is isomorphic to the associated matroid to the {\it geometric dual graph} $G^*$ (see  \cite[Proposition 5.2.1]{Ox}). Well-known examples of non-planar graphs are complete graph $K_5$ on 5 vertices and a special bipartite graph $K_{3,3}$. 
In fact, by Kuratowski's theorem, a graph is planar if and only if it has no minor isomorphic to $K_5$ or $K_{3,3}$ (cf.\ \cite[Theorem 2.3.8]{Ox}).  
\item[(3)]Many regular matroids are graphic or cographic. Actually, by Seymour's decomposition theorem (\cite{Sey}), all regular matroids are obtained as ``sums'' (precisely, 1-sum, 2-sum ,and 3-sum) of some graphic matroids, cographic matroids and an exceptional matroid $R_{10}$.  
\end{itemize}
\end{remark}

The necessary and sufficient condition for two graphic matroids to be isomorphic to each other is given as the following. 
\begin{theorem}{\rm (Whitney's 2-isomorphism theorem cf.\ \cite[Theorem 5.3.1]{Ox})}\label{thm:Whitney}

Let $G_1$ and $G_2$ be two finite graphs (without isolated vertices). Then the following are equivalent: 
\begin{itemize}
\item[(1)]$M(G_1)\cong M(G_2)$. 
\item[(2)]$G_2$ can be transformed into $G_1$ by a sequence of the following two operations (we describe these operations by the example below, for the more precise definition, see \cite[section 5.3]{Ox})
\begin{itemize}
\item[(i)]The vertex identification and the vertex cleaving.
\item[(ii)]The Whitney twist
\end{itemize}
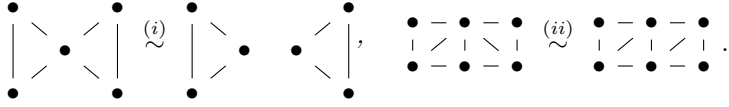
\begin{minipage}{0.3\hsize}
\[\begin{tikzcd}[row sep=tiny, column sep=tiny]
\bullet\ar[dd, no head]\ar[dr, no head]&&\bullet\ar[dd, no head]\ar[dl, no head]\\
&\bullet&\\
\bullet\ar[ur, no head]&&\bullet\ar[ul, no head]
\end{tikzcd}\overset{(i)}{\sim}
\begin{tikzcd}[row sep=tiny, column sep=tiny]
\bullet\ar[dd, no head]\ar[dr, no head]&&&\bullet\ar[dd, no head]\ar[dl, no head]\\
&\bullet&\bullet&\\
\bullet\ar[ur, no head]&&&\bullet\ar[ul, no head]
\end{tikzcd}\]
\end{minipage}
\begin{minipage}{0.1\hsize}
\ \ \ \ \ \ \ {,} 
\end{minipage}
\begin{minipage}{0.3\hsize}
\[\begin{tikzcd}[row sep=tiny, column sep=tiny]
\bullet\ar[d, no head]\ar[r, no head]&\bullet\ar[dl, no head]\ar[d, no head]\ar[dr, no head]\ar[r, no head]&\bullet\ar[d, no head]\\
\bullet\ar[r, no head]&\bullet\ar[r, no head]&\bullet\\
\end{tikzcd}\overset{(ii)}{\sim}
\begin{tikzcd}[row sep=tiny, column sep=tiny]
\bullet\ar[d, no head]\ar[r, no head]&\bullet\ar[dl, no head]\ar[d, no head]\ar[r, no head]&\bullet\ar[d, no head]\ar[dl, no head]\\
\bullet\ar[r, no head]&\bullet\ar[r, no head]&\bullet\\
\end{tikzcd}.
\]
\end{minipage}
\end{itemize}  
\end{theorem}

\section{The universal Poisson deformation space of hypertoric varieties}
In this section, we will construct explicitely the universal Poisson deformation space of a smooth (resp. affine) hypertoric variety $Y(A, \alpha)$ (resp. $Y(A, 0)$) associated to a unimodular matrix $A$. 
First, we review (the universal) Poisson deformation space of general symplectic varieties, and we state a Namikawa's result which claims the existence of the universal Poisson deformation space of conical symplectic varieties and  its projective crepant resolutions. Also, these universal Poisson deformation spaces admit a certain commutative diagram (cf.\ Theorem \ref{Namikawa}). 
For a Poisson variety $(Y, \{-, -\}_0)$ and an affine scheme $(B, 0)$ with fixed point 0, we call a Poisson $B$-scheme $(\calY, \{-, -\})$ a {\it Poisson deformation (space)} of $Y$ if $\calY\to B$ is flat, its each fiber is a Poisson scheme, and the central fiber is isomorphic to $(Y, \{-, -\}_0)$ as a Poisson variety. 
Next, we define the universal Poisson deformation.  
\begin{definition}{\rm (The universal Poisson deformation)}\\
A Poisson deformation $(\calY, \{-, -\})\to B$ of a Poisson variety $(Y, \{-, -\}_0)$ is {\it universal at 0} if the following holds:
 
For each infinitesimal Poisson deformation $(\calX, \{-, -\}')\to\Spec A$ of $Y$, i.e., $(A, \germ_A)$ is an Artinian algebra with the residue field $\C$, there exists a unique morphism $f: \Spec A\to B$ such that $f(\germ_A)=0$ and 
\[\begin{tikzcd}
\calX\ar[r]\ar[d]&\calY\ar[d]\\
\Spec A\ar[r, "f"]&B
\end{tikzcd}\]
is Cartesian.
\end{definition}   
If the universal Poisson deformation exists, the completion of it at $0$ is unique up to isomorphism. For a conical symplectic variety $Y_0$ and its projective symplectic resolution $\pi: Y\to Y_0$, the existence of the universal Poisson deformation space of them is known by Namikawa as below. Moreover, it is known that $\Cstar$-action on $Y_0$ is uniquely extended to $Y$ (cf.\ \cite[Proposition A.7]{NamFlops}). Then the completion of a universal Poisson deformation is uniquely algebraized to $\calY\to H^2(Y, \C)$ (resp. $\calY_0\to H^2(Y, \C)$), i.e., this is the $\Cstar$-equivariant universal Poisson deformation (for the detail, see \cite[section 4]{NamFlops} and \cite[section 5.4]{NamPDaffine}). In this paper, we say $\calY\to H^2(Y, \C)$ (resp. $\calY_0\to H^2(Y, \C)$) is the universal Poisson deformation of $Y$ (resp. $Y_0$). Namikawa showed {that} two universal deformation spaces  $\calY\to H^2(Y, \C)$ and $\calY_0\to H^2(Y, \C)$ are related to each other as the following. 
\begin{theorem}{\rm (Namikawa \cite{NamPDbir})}\label{Namikawa}\\
Let $Y_0$ be a conical symplectic variety and we assume that there exists a projective symplectic resolution $\pi : Y \to Y_0$.
Then there exists the universal Poisson deformation space $\calY\to H^2(Y, \C)$ (resp. $\calY_0\to H^2(Y, \C)/W$) of $Y$ (resp. $Y_0$), and they satisfy the following $\Cstar$-commutative diagram.  
\[
\begin{tikzcd}[row sep=tiny, column sep=tiny, contains/.style = {draw=none,"\in" description,sloped}, icontains/.style = {draw=none,"\ni" description,sloped}]
&Y\ar[dl, hook']\ar[rr, "\pi"]\ar[dd]&&Y_0\ar[dl, hook'] \ar[dd]\\
\calY\ar[rr, crossing over, "\Pi" near end]\ar[dd, "\overline{\mu}"]&&\calY_0&\\
&0\ar[dl, icontains]\ar[rr, mapsto]&&\overline{0}\ar[dl, icontains]\\
H^2(Y, \C)\ar[rr, "\psi"]&&H^2(Y, \C)/W \ar[from=uu, crossing over, "\overline{\mu}_W" near start]&
\end{tikzcd},
\]
where $\psi$ is the Galois cover by a certain finite group $W$, which acts linearly on $H^2(Y, \C)$. 
\end{theorem}

We call $W$ the {\it Namikawa--Weyl group} of $Y_0$. As described by Namikawa in \cite[Section1]{NamWeylgrp}, we can describe $W$ explicitly in terms of the exceptional sets of $\pi$ as below. First, by \cite{Kalstratify}, the singular locus $(Y_0)_{\Sing}$ is stratified by smooth symplectic varieties. Let $\Sigma_{\codim\geq4}$ denote the union of strata of codimension 4 or higher, and define $\Sigma_{\codim2}:=(Y_0)_{\Sing}\setminus\Sigma_{\codim\geq4}$. Then, for each component $Z_k$ of the connected component decomposition $\Sigma_{\codim2}=\bigsqcup_{k=1}^{s}{Z_k}$, one can consider a transversal slice $S_{\ell_k}$ through a point $x\in Z_k$. 
Since $S_{\ell_k}=S_{\Delta_{\ell_k}}$ is a symplectic surface, i.e., the $ADE$ type surface singularity with the corresponding Dynkin diagram $\Delta_{\ell_k}$, so $\pi:Y\to Y_0$ is locally (at $x$) isomorphic to $p\times id: \tilde{S}_{\ell_k}\times\C^{2m-2}\to S_{\ell_k}\times\C^{2m-2}$, where $2m=\dim Y_0$ and $p$ is the minimal resolution of $S_{\ell_k}$.   
We consider all $(-2)$-curves $C_i \ (1 \leq i \leq \ell_k)$ in $\tilde{S}_{\ell_k}$ and set
\[
\Phi_{\ell_k}:=\Set{\sum_{i=1}^{\ell_k}{d_i[C_i]} \ | \ d_i \in \Z \ \text{s.t.} \ \left(\sum_{i=1}^{\ell_k}{d_i[C_i]}\right)^2=-2} \subset H^2(\tilde{S}_{\ell_k}, \R).
\]
Then, $\Phi_{\ell_k}$ defines the corresponding $ADE$ type root system in $H^2(\tilde{S}_{\ell_k}, \R)$, and the  associated usual Weyl group $W_{S_{\ell_k}}$ acts on $H^2(\tilde{S}_{\ell_k}, \R)$. 
However this description is local at each point on $Z_k$, and the number of irreducible components of $\pi^{-1}(Z_k)$ may be less than $\ell_k$ globally. In fact, the following homomorphism is defined by the monodromy: 
\[
\rho_k : \pi_1(Z_k) \to \Aut(\Delta_{\ell_k}),
\]
where $\Delta_{\ell_k}$ is the associated Dynkin diagram and $\Aut(\Delta_{\ell_k})$ is its graph automorphism group. 
Then, we can define the subgroup of $W_{S_{\ell_k}}$ as the following: 
\[
W_{Z_k}:=W_{S_{\ell_k}}^{\Image\rho_k}:=\{\sigma \in W_{S_{\ell_k}} \ | \ \sigma\iota=\iota\sigma \ (\iota \in \Image\rho_k)\}.
\] 
Finally, taking the direct product of them, we get the Namikawa--Weyl group.  
\[
W:=\prod_{k}W_{Z_k}.
\]  

The aim of this section is to determine the commutative diagram in Theorem \ref{Namikawa} for the symplectic resolution $\pi: Y(A, \alpha) \to Y(A, 0)$ of each affine hypertoric variety $Y(A, 0)$ (cf.\ Theorem \ref{thm:Mainthm}).
\begin{proposition}For a generic $\alpha$, $\overline{\mu}_\alpha: X(A, \alpha)\to \C^d$ and $\overline{\mu}_0: X(A, 0)\to \C^d$ are Poisson deformations of $Y(A, \alpha)$ and $Y(A, 0)$ respectively. 
\end{proposition}
\begin{proof}
First, we show that $\overline{\mu}_\alpha: X(A, \alpha)\to\C^d$  is flat. By Corollary \ref{cor:hypertoricsymplectic}, each fiber $Y(A, (\alpha, \xi))$ of $\overline{\mu}_\alpha$ has the same dimension $2n-2d$. Then, by the flat criterion \cite[Corollary 14.128.]{GW} and smoothness of $X(A, \alpha)$, one can show that $\overline{\mu}_\alpha$ is flat. 
Next, we show that $\overline{\mu}_0: X(A, \alpha)\to\C^d$  is flat. By Corollary \ref{cor:hypertoricsymplectic} and Proposition \ref{prop:birational}, each fiber $Y(A, (\alpha, \xi))$ of $\overline{\mu}_\alpha$ has the same dimension $2n-2d$. By the Hochster's theorem (\cite{HR}), $X(A, 0)=\Spec\C[\bm{z}, \bm{w}]^{\T_\C^d}$ is Cohen--Macaulay, so we can also apply the flat criterion to this case and prove the flatness of $\overline{\mu}_0$.  
Moreover, by Theorem \ref{thm:hyperfund}, each fiber $Y(A, (\alpha, \xi))$ (resp. $Y(A, (0, \xi))$) of $\overline{\mu}_{\alpha}$ (resp. $\overline{\mu}_0$) is a symplectic variety, in particular a Poisson variety. Thus, $\overline{\mu}_\alpha$ (resp.  $\overline{\mu}_0$) is a Poisson deformation of $Y(A, \alpha)$ (resp. $Y(A, 0)$). 
\end{proof}
Now, we have a (smooth) Lawrence toric variety $X(A, \alpha)$ and a flat map $\overline{\mu}_\alpha: X(A, \alpha) \to \C^d$ which is a Poisson deformation of $Y(A, \alpha)$. Although it may be well-known for experts that $\overline{\mu}_\alpha$ gives the universal Poisson deformation of $Y(A, \alpha)$ (actually this fact is mentioned in \cite[Example 2.1]{BPW1} without proof), we explain a proof.
First, each fiber $(({\mu}^{-1}(\xi))^{\alpha-st}/\T_\C^d, \ \om_\xi)$ of $\overline{\mu}_\alpha$ is also a smooth symplectic variety. Since we have a conical $\Cstar$-action {of positive weight} (cf.\ Definition \ref{def:conical} (i)), we can apply the Slodowy's lemma \cite[Remark 4.3]{Slo} to $\overline{\mu}_\alpha$, and we can deduce $\overline{\mu}_\alpha : X(A, \alpha)\to \C^d$ is a trivial $C^{\infty}$-fiber bundle. Then by using this trivialization, one can easily show that the Kodaira-Spencer map associated to $\overline{\mu}_\alpha$ is the same as the differential $(dp)_0$ at 0 of the associated period map $p: \C^d\to H^2(Y(A, \alpha), \C) : \xi \mapsto [\om_\xi]$. In particular, $\overline{\mu}_\alpha$ is the universal Poisson deformation space of $Y(A, \alpha)$ if and only if $(dp)_0$ is an isomorphism. Then, by the algebro-geometric version of the Duistermaat--Heckman theorem (\cite[Proposition 3.2.1]{Lo}), the period map $p$ is equal to the complexification $\kappa_2\otimes_\Z\C$ of so-called the Kirwan map $\kappa_2 : \Z^d \to H^2(Y(A, \alpha), \Z)$. The Kirwan map $\kappa_2$ is defined as $\xi \mapsto c_1(L_\xi)$, where $L_\xi=\mu^{-1}(0)^{\alpha-st}\times_{\T_\C^d}\C:=(\mu^{-1}(0)^{\alpha-st}\times\C)/\T_\C^d$ is the associated line bundle to the character $\xi\in\Z^d=\Hom(\T_\C^d, \Cstar)$. Then, since the Kirwan map is linear, we only have to show that the Kirwan map is an isomorphism.    
In \cite{Kocohomology}, Konno showed $\kappa_2$ is always an isomorphism under the condition that $B$ doesn't have any zero row vectors. As a consequence, we obtain the following. 
\begin{proposition}\label{prop:UPDforsmooth}
Let $\alpha\in\Z^d$ be a generic element. Assume $B$ doesn't have any zero row vectors. Then, the map $\overline{\mu}_\alpha : X(A, \alpha) \to \C^d$ gives the univeral Poisson deformation space at 0 of a smooth hypertoric variety $Y(A, \alpha)$.
\end{proposition} 
\begin{remark}
As stated in Corollary \ref{cor:Breduction}, we can always assume that $B$ doesn't have any zero row vectors without changing the $\Cstar\times\T_\C^{n-d}$ isomorphism class of $Y(A, \alpha)$ as Poisson varieties. 
\end{remark}

Next, we will determine the universal Poisson deformation space of $Y(A, 0)$. 
First, we describe the Namikawa--Weyl group of $Y(A, 0)$ by analyzing the structure of codimension 2 singular locus $\Sigma_{\codim2}$.  
To state a known result on a stratification of $Y(A, 0)$ in \cite{PW}, we prepare some notations of hyperplane arrangements. We call $\calH$ a {\it central multi-arrangement} in $\R^{n-d}$ if $\calH=\{H_1, \ldots ,H_n\}$ is a multiset of hyperplanes through the origin in $\R^{n-d}$. A subset $F\subseteq\{1,2,\ldots,n\}$ is called a {\it flat} if it satisfies $F=\{i \ | \ \bigcap_{j\in F}{H_j}\subseteq H_i\}$. For a flat $F$, we set $H_F:=\cap_{j\in F}{H_j}$ and define the {\it rank} of $F$ as $\rank F:={\codim} H_F$. In particular, each rank 1 flat is  corresponding to some indices set of parallel hyperplanes in $\calH$.  
From a matrix $B$, we can construct a central multi-arrangement $\calH_B$ by collecting hyperplanes $H_{\bm{b_i}}:=\{\bm{v}\in\R^{n-d} \ | \ \langle\bm{v}, \bm{b_i}\rangle=0\}$, where $\langle-,-\rangle$ is the standard inner product in $\R^{n-d}$. 

For a flat $F$, we consider the following exact sequence induced from\\ $\begin{tikzcd}[column sep=small]0\ar[r]&\Z^{n-d}\ar[r, "B"]&\Z^n\ar[r, "A"]&\Z^d\ar[r]&0\end{tikzcd}$ and the natural projection $\Z^n\twoheadrightarrow\Z^{|F|}$:
 \[
\begin{tikzcd}
0\ar[r]&\Z^{n-d}/H_F\ar[r, "{B}_F"]&\Z^{|F|}\ar[r, "{A}_F"]&\Z^d/\sum_{i \notin F}\Z\bm{a_i}\ar[r]&0
\end{tikzcd}.
\]
In \cite{PW}, for any $A$, they give a stratification of $Y(A, 0)$ by smooth strata. Also, they describe the slice of each stratum as the following. 
\begin{theorem}{\rm (\cite[Lemma 2.5]{PW})}\\
In the setting above, 
\[Y(A, 0)=\bigsqcup_{F: \text{flat}}{\mathring{Y}(A, 0)^F}\]
is a stratification by $2(n-d-\rank F)$ dimensional smooth strata $\mathring{Y}(A, 0)^F$ defined as the following: 
\[\mathring{Y}(A, 0)^F:=\{(\bm{z}, \bm{w}) \in \mu^{-1}(0) \ | \ (z_j, w_j)=0 \Leftrightarrow j \in F\}/\T_\C^d.\]
Moreover, the slice of each stratum $\mathring{Y}(A, 0)^F$ is $Y(A_F, 0)$. 
\end{theorem}
To describe explicitly the codimension 2 singular locus $\Sigma_{\codim2}$ of $Y(A, 0)$, we interchange the row vectors $\bm{b_i}$ of $B$, and multiplying $\bm{b_i}$ by $\pm1$ if necessarily. Then, we can transform $B$ as the following form: 
\[\arraycolsep=4pt
B=\left(
\begin{array}{ccc}
&\text{\huge{$B^{(1)}$}}&\\
\hline
&\text{\huge{$B^{(2)}$}}&\\
\hline
&\vdots&\\
\hline
&\text{\huge{$B^{(s)}$}}&\\
\end{array}
\right), \ \ \ 
B^{(k)}=\begin{array}{rcccll}
	\ldelim({3}{4pt}[] &&\bm{b^{(k)}}&&\rdelim){3}{4pt}[]&\rdelim\}{3}{10pt}[$\ell_k$] \\
	&&\vdots&& \\
	&& \bm{b^{(k)}} && \\
	\end{array}\ \ , \arraycolsep=2pt
\]
where $\bm{b^{(k_1)}}\neq\pm\bm{b^{(k_2)}}$ if $k_1\neq k_2$. 
Since $B$ is unimodular, if  $k_1\neq k_2$, then $\bm{b^{(k_1)}}\neq\ell\bm{b^{(k_2)}}$ for any $\ell\in\Z\setminus\{0\}$. This means that each rank 1 flat of $\calH_B$ is given by the form $F_k:=\{m_{k-1}+1, \ldots ,m_k\}$, where $m_k:=\sum_{i=1}^k{\ell_i}$. Now we can describe the codimension 2 singular locus $\Sigma_{\codim2}$ of $Y(A, 0)$ as the following. 
\begin{corollary}\label{cor:codim2strata}
In the setting above, the codimension 2 singular locus $\Sigma_{\codim2}$ of $Y(A, 0)$ is 
\[\Sigma_{\codim2}=\bigsqcup_{k : \ell_k\geq2 }{\mathring{Y}(A, 0)^{F_k}}.\]
Additionally, the slice of $\mathring{Y}(A, 0)^{F_k}$ is the $A_{\ell_k-1}$ type surface singularity. In particular, the Namikawa--Weyl group $W$ of $Y(A, 0)$ is a subgroup of $W_B:=\gerS_{\ell_1} \times \cdots \times \gerS_{\ell_s}$. 
\end{corollary}
\begin{proof}
By the theorem above, $\Sigma_{\codim2}\subseteq\bigsqcup_{k=1}^s{\mathring{Y}(A, 0)^{F_k}}$, and the slice of each $\mathring{Y}(A, 0)^{F_k}$ is given by $Y(A_{{F_k}}, 0)$. By definition and $H_{F_k}=H_{\bm{b^{(k)}}}$, $Y(A_{{F_k}}, 0)$ is the associated affine hypertoric variety to the following exact sequence:
\[\begin{tikzcd}
0\ar[r]&\Z^{n-d}/H_{\bm{b^{(k)}}}\ar[r, "{B}_{{F_k}}"]&\Z^{|{F_k}|}\ar[r, "{A}_{{F_k}}"]&\Z^d/\sum_{i \notin {F_k}}\Z\bm{a_i}\ar[r]&0
\end{tikzcd}.\]  
Then, $B_{{F_k}}$ can be represented as $(1,1, \ldots ,1)^T$ by an appropriate basis. Thus by Example \ref{Amsurf}, the corresponding affine hypertoric variety $Y(A_{F_k}, 0)$ is the $A_{\ell_k-1}$ type surface singularity, where if $\ell_k$=1, $Y(A_{F_k}, 0):=\C^2$. 
This proves $\Sigma_{\codim2}=\bigsqcup_{k : \ell_k\geq2 }{\mathring{Y}(A, 0)^{F_k}}$. In addition, by the definition of the Namikawa--Weyl group, we have $W\subseteq W_B$.  
\end{proof}

Below, we will show $W=W_B$. As remarked in Remark \ref{rem:comm diagram}, we have the following $\Cstar$-equivariant commutative diagram: 
\[
\begin{tikzcd}[row sep=tiny, column sep=tiny, contains/.style = {draw=none,"\in" description,sloped}, icontains/.style = {draw=none,"\ni" description,sloped}]
&Y(A, \alpha)\ar[dl, hook']\ar[rr, "\pi"]\ar[dd]&&Y(A, 0)\ar[dl, hook'] \ar[dd]\\
X(A, \alpha)\ar[rr, crossing over, "\Pi" near end]\ar[dd, "\overline{\mu}_\alpha"]&&X(A, 0)&\\
&0\ar[dl, contains]\ar[rr, mapsto]&&0\ar[dl, contains]\\
\C^d\ar[rr, equal]\ar[ur, icontains]&&\C^d \ar[from=uu, crossing over, "\overline{\mu}_0"near start]\ar[ur, icontains]&
\end{tikzcd}.
\]
As the following, if one can construct a good $W_B$-action on $X(A, 0)$ and $\C^d$, then one can show $W=W_B$ and construct the universal Poisson deformation of $Y(A, 0)$. 
\begin{lemma}\label{Wactionlem}
In the setting above, suppose that we could construct a $W_B$-action on $X(A, 0)$ and $\C^d$ satisfying the following properties:    
\begin{itemize}
\item[(i)]The $W_B$-action on $\C^d$ is linear {and effective}. Also, the $W_B$-action on $X(A, 0)$ preserves its Poisson structure.  
\item[(ii)]The $W_B$-action commutes with $\Cstar$-action on $\C^d$ and $X(A, 0)$. Also, $\overline{\mu}_0: X(A, 0) \to \C^d$ is $W_B$-equivariant. 
\item[(iii)]The induced $W_B$-action on $Y(A, 0) \subseteq X(A, 0)$ is trivial.  
\end{itemize}
Then, the induced $\Cstar$-equivariant morphism $\overline{\mu}_{W_B}: X(A, 0)/W_B \to \C^d/W_B$ gives the universal Poisson deformation space of $Y(A, 0)$.  

\end{lemma}
\begin{proof}
Let $\calX \to \C^d/W$ be the universal Poisson deformation space of $Y(A, 0)$. By the assumption and Theorem \ref{Namikawa}, we have the following $\Cstar$-equivariant commutative diagram. 
\[\setlength{\abovedisplayskip}{-3pt}
\begin{tikzcd}[row sep=tiny, column sep=tiny]
&X(A, 0)/W_B\ar[dd]&\\
X(A, \alpha)\ar[ur, "\Pi"]\ar[rr, crossing over]\ar[dd]&&\calX\ar[dd]\\
&\C^d/W_B&\\
\C^d\ar[rr, "\psi"]\ar[ur, "\psi_{B}"]&&\C^d/W
\end{tikzcd}, \setlength{\belowdisplayskip}{-3pt}
\]
where $\psi_{{B}}$ is the quotient map by $W_B$. 
Then, if one takes the completion of the diagram above at 0, then by the universality of $\calX$ at 0 $\in\C^d/W$, the following diagram will be induced and $\widehat{X}(A, 0)\cong \widehat{\calX}\times_{\widehat{\C^d/W_B}}{\widehat{\C^d/W}}$:  
\[\setlength{\abovedisplayskip}{-3pt}
\begin{tikzcd}[row sep=tiny, column sep=tiny]
&\widehat{X(A, 0)/W_B}\ar[dd]\ar[dr, dashrightarrow]&\\
\widehat{X(A, \alpha)}\ar[ur, "\widehat{\Pi}"]\ar[rr, crossing over]\ar[dd]&&\widehat{\calX}\ar[dd]\\
&\widehat{\C^d/W_B}\ar[dr, dashrightarrow]&\\
\widehat{\C^d}\ar[rr, "\widehat{\psi}"]\ar[ur, "\widehat{\psi}_{B}"]&&\widehat{\C^d/W}
\end{tikzcd}. \setlength{\belowdisplayskip}{-3pt}
\]
Since $\widehat{\psi}_B, \ \widehat{\psi}$ are $\Cstar$-equivariant, $\widehat{\C^d/W_B} \to \widehat{\C^d/W}$ is also $\Cstar$-equivariant. Then, using this $\Cstar$-action, one can algebraize this diagram as the following diagram (cf.\ \cite[section 4]{NamFlops} and \cite[section 5.4]{NamPDaffine}), that is, the diagram above is the same as the completion of the following $\Cstar$-equivariant diagram at 0:  
\[
\begin{tikzcd}[row sep=tiny, column sep=tiny]
&X(A, 0)/W_B\ar[dd]\ar[dr]&\\
X(A, \alpha)\ar[ur, "\Pi"]\ar[rr, crossing over]\ar[dd]&&\calX\ar[dd]\\
&\C^d/W_B\ar[dr]&\\
\C^d\ar[rr, "{\psi}"]\ar[ur, "\psi_{B}"]&&\C^d/W
\end{tikzcd}.
\]
Now, we have $W\subseteq W_B\subseteq GL(\C^d)$ and $\C^d/W_B \to \C/W$ as above. Hence, $W=W_B$ and this completes the proof.  
\end{proof}

Below, we will construct a $W_B$-action satisfying the  properties in the Lemma \ref{Wactionlem}.  
As the previous setting, we can assume $B$ is in the following form: 
\begin{equation}\tag{*}\arraycolsep=4pt
B=\left(
\begin{array}{ccc}
&\text{\huge{$B^{(1)}$}}&\\
\hline
&\text{\huge{$B^{(2)}$}}&\\
\hline
&\vdots&\\
\hline
&\text{\huge{$B^{(s)}$}}&\\
\end{array}
\right), \ \ \ 
B^{(k)}=\begin{array}{rcccll}
	\ldelim({3}{4pt}[] &&\bm{b^{(k)}}&&\rdelim){3}{4pt}[]&\rdelim\}{3}{10pt}[$\ell_k$] \\
	&&\vdots&& \\
	&& \bm{b^{(k)}} && \\
	\end{array}\ \ \ \ , \arraycolsep=2pt
\end{equation}
where if $k_1 \neq k_2$, then $\bm{b^{(k_1)}}\neq\pm\bm{b^{(k_2)}}$. Under this assumption, we can describe and construct the $W_B$-action.
Note that $W_B:=\gerS_{\ell_1} \times \cdots \times \gerS_{\ell_s}\subseteq \gerS_n$ naturally acts on $\C^{2n}$ (resp. $\C^n$) as the coordinate permutation $z_i\mapsto z_{\sigma(i)}$, $w_i\mapsto w_{\sigma(i)}$ (resp. $u_i\mapsto u_{\sigma(i)}$), where $(z_1, \ldots, z_n, w_1, \ldots ,w_n)$ (resp. $(u_1, \ldots u_n)$) is the coordinate of $\C^{2n}$ (resp. $\C^n$).

Then we can construct the desired $W_B$-action as the following. 
\begin{proposition}\label{prop:HT-NWgrp}
In the setting above, for $W_B:=\gerS_{\ell_1}\times\cdots\times\gerS_{\ell_s}$, 
\begin{itemize}
\item[$\bullet$] The $W_B$-action on $\C^{2n}$ induces a $W_B$-action on $X(A, 0)=\C^{2n}/\hspace{-3pt}/\T_\C^d$.
\item[$\bullet$] The $W_B$-action on $\C^n$ induces a $W_B$-action on $\C^d$ ($\bm{a_i} \mapsto \bm{a_{\sigma(i)}}$) through the surjection $A: \C^n\to\C^d$.
\end{itemize}
These actions satisfy the properties in Lemma \ref{Wactionlem}. 
Moreover, we have the following (cf.\ Lemma \ref{ring str lem})
\[\C[X(A, 0)/W_B]=\C[z_1w_1, \ldots ,z_nw_n]^{W_B}[f_{\bm{\beta}} \ | \ \bm{\beta}\in \Image B].\] 
\end{proposition}
\begin{remark}
By definition, the $W_B$-action on $\C^{2n}$ does not commute with the $\T_\C^d$-action on it in general. 
\end{remark}
\begin{proof}
First, we show that the $W_B$-action on $X(A, 0)$ and $\C^d$ is well-defined. Note that $\mu: \C^{2n}\to \C^d$ is decomposed as $\C^{2n}\xrightarrow{\Psi}\C^n\xrightarrow{A}\C^d$, 
where $\Psi(\bm{z}, \bm{w})=\sum_{i=1}^n{z_iw_i\bm{e_i}}$ ($\Psi$ is the moment map associated to the natural $\T_\C^n$-action on $\C^{2n}$). Since $\Psi$ and $\mu$ are $\T_\C^d$-invariant, the following diagram is induced. 
\begin{equation}\tag{**}
\begin{tikzcd}
     \C^{2n}    \ar[d, two heads]                                                                            &           \\
     X(A, 0) \ar[dr, two heads, "\overline{\mu}"] \ar[d, two heads, "\overline{\Psi}"]         &            \\
     \C^n \ar[r, two heads,  "A"]                                                                         & \C^d    \\
\end{tikzcd}\ \ \ \  \ \ \ \begin{tikzcd}
    \C[\bm{z}, \bm{w}]                                       &  \\
    \C[\bm{z}, \bm{w}]^{\T_\C^d} \ar[u, hook]         &   \\
    \C[u_1, \ldots ,u_n] \ar[u, hook, "{\overline{\Psi}}^*"]    & \C[t_1, \ldots ,t_d] \ar[l, hook',  "{}^t\!A"] \ar[ul, hook', "{\overline{\mu}}^*"]   \\
\end{tikzcd}, \setlength{\belowdisplayskip}{-8pt} 
\end{equation}
where the right commutative diagram is the diagram of the corresponding coordinate rings to the left diagram.  
It is clear that the $W_B$-action on $\C^n$ as the permutation of the coordinates is the same as the induced action from $W_B$-action on $\C^{2n}$ by the above inclusion $\C[u_1, \ldots ,u_n] \hookrightarrow \C[\bm{z}, \bm{w}]$. Then to show that the desired $W_B$-action on $X(A, 0)$ and $\C^d$ is well-defined, we only have to show that the $W_B$-action on $\C[\bm{z}, \bm{w}]$ preserves all subalgebras appeared in the diagram above. 
For $\C[X(A, 0)]=\C[\bm{z}, \bm{w}]^{\T_\C^d}$, by Lemma \ref{ring str lem}, we have the following description: 
\[\C[\bm{z}, \bm{w}]^{\T_\C^d}=\C[z_1w_1, \ldots ,z_nw_n][f_{\bm{\beta}} \ | \ \bm{\beta}\in \Image B].\]
Since $B$ is in the form of (*), each element $f_{\bm{\beta}} \ (\beta\in\Image B)$ is in the following form: 
\begin{align*}
f_{\bm{\beta}}&={\prod_{k: \beta^{(k)}>0}{\left(\prod_{m\in F_k}{z_{m}}\right)^{\beta^{(k)}}}\prod_{k: \beta^{(k)}<0}{\left(\prod_{m\in F_k}{w_m}\right)^{-\beta^{(k)}}}}
\end{align*} 
where $\beta^{(k)}:=\beta_m \ (m\in F_k)$ (for the definition of $F_k$, see Corollary \ref{cor:codim2strata}) 
Then, each $f_{\bm{\beta}}$ is invariant under the $W_B$-action ($z_iw_i\mapsto z_{\sigma(i)}w_{\sigma(i)} \ (\sigma\in W_B)$). It implies that the $W_B$-action preserves $\C[\bm{z}, \bm{w}]^{\T_\C^d}$. In particular, 
\begin{equation}\tag{***}
\C[X(A, 0)/W_B]=\C[z_1w_1, \ldots ,z_nw_n]^{W_B}[f_{\bm{\beta}} \ | \ \bm{\beta}\in \Image B].\end{equation}
It is easily checked that $W_B$ also preserves $\C[t_1, \ldots ,t_d]$ (or directly, as similar discussion above, one can prove this by using {the $W_B$-invariance of $\Image B=\Ker A$}). This completes the proof of the well-definedness of the $W_B$-action on $X(A, 0)$ and $\C^d$. 

Next, we will show that this $W_B$-action has the properties (i), (ii), and (iii) stated in Lemma \ref{Wactionlem}.

\noindent(i) \ The linearity of the $W_B$-action on $\C^d$ is obvious. {Also, $W_B:=\prod_{k=1}^{s}{\gerS_{\ell_k}}$-action is effective. In fact, if not so, there exists distinct $i, j\in F_k$ for some $1\leq k\leq s$ such that $\bm{a_i}=\bm{a_j}$. In particular, this implies that $\bm{e_i}-\bm{e_j}\in \Ker A=\Image B$. However, this cannot be happened since $\bm{b_i}=\bm{b_j}=\bm{b^{(k)}}\neq0$.} Finally, since the $W_B$-action preserves the symplectic structure of $\C^{2n}$ and the Poisson structure of $X(A, 0)$ is induced from $\C^{2n}$, the $W_B$-action on $X(A, 0)$ preserves the Poisson structure as well. 

\noindent(ii) \ The commutativity of the $W_B$-action with the $\Cstar$-action is obvious. The $W_B$-equivariance of $\overline{\mu}$ is also clear by the above proof of the well-definedness. 

\noindent(iii) \ By definition, $\C[Y(A, 0)]=\C[X(A, 0)]/J_\mu$, where $J_\mu:=\left\langle \displaystyle\sum_{j=1}^n{z_jw_ja_{ij}} \ \left|\right. \ i=1, \ldots ,d\right\rangle$. Then, 
\[J_\mu=\left\langle\sum_{j=1}^n{z_jw_j\alpha_j} \ \left|\right. \bm{\alpha}=(\alpha_1, \ldots ,\alpha_n) \in \Image({}^t\!A) \right\rangle=\left\langle\sum_{j=1}^n{z_jw_j\alpha_j} \ \left|\right. \bm{\alpha}\in \Ker(B^T) \right\rangle.\]
Thus by a similar discussion as above, $J_\mu$ is preserved by the $W_B$-action. Then, the $W_B$-action on $Y(A, 0)$ is induced. 
To show that this action is trivial, by (***), we only have to show that for each $i$, $z_iw_i-z_{\sigma(i)}w_{\sigma(i)}\in J_\mu$. 
This is followed from that $\bm{e_i}-\bm{e_{\sigma(i)}}\in\Ker(B^T)$ by $\bm{b_i}=\bm{b_{\sigma(i)}}$. 
This completes the proof.
\end{proof}

We can summarize the main result of this section as the following.  
\begin{theorem}\label{thm:Mainthm}
Let $A$ be a unimodular matrix  and $\alpha\in\Z^d$ be a generic element. If for $B$, $\bm{b_i}\neq0$ ($1\leq i\leq n$) and we take $B$ as (*), then the diagram of Theorem \ref{Namikawa} for the affine hypertoric variety $Y(A, 0)$ is obtained as the following:  
\[
\begin{tikzcd}[row sep=tiny, column sep=tiny, contains/.style = {draw=none,"\in" description,sloped}, icontains/.style = {draw=none,"\ni" description,sloped}]
&Y(A, \alpha)\ar[dl, hook']\ar[rr, "\pi"]\ar[dd]&&Y(A, 0)\ar[dl, hook'] \ar[dd]\\
X(A, \alpha)\ar[rr, crossing over, "\Pi_{W_B}" near end]\ar[dd, "\overline{\mu}_\alpha"]&&X(A, 0)/W_B&\\
&0\ar[rr, mapsto]&&\overline{0}\\
\C^d\ar[rr, "\psi"]\ar[ur, icontains]&&\C^d/W_B\ar[ur, icontains] \ar[from=uu, crossing over, "\overline{\mu}_{W_B}" near start]&
\end{tikzcd},
\]
where $\Pi_{W_B}$ is the composition of $\Pi: X(A, \alpha) \to X(A, 0)$ and the quotient map of $X(A, 0)$ by $W_B=\gerS_{\ell_1}\times\cdots\times\gerS_{\ell_s}$. 
\end{theorem}

\section{Classification of affine hypertoric varieties}
In this section, we prove for unimodular matrices $A$ and $A'$, $Y(A, 0)\cong Y(A', 0)$ as conical symplectic varieties if and only if $A\sim A'$ i.e., these are transformed {to} each other by some operations (cf.\ Definition \ref{def:fund}). As a corollary, we obtain a criterion when two quiver varieties whose dimension vector has all coordinates equal to one are isomorphic to each other. Then we describe all 4- and 6-dimensional affine hypertoric varieties (in these cases, all of them are obtained as quiver varieties). Also, as a related result, we will prove if $Y(A, 0)$ is isomorphic to a nilpotent orbit closure of some semisimple Lie algebra, then $Y(A, 0)$ is the direct product $\overline{\calO}=\overline{{\calO}^{\tmin}_{A_{\ell_1}}}\times\cdots\times\overline{{\calO}^{\tmin}_{A_{\ell_s}}}$ of some minimal nilpotent orbit closures. 

First, we refer the following result by Arbo and Proudfoot (\cite{AP}) such that $\T_\C^{n-d}$-equivariant isomorphism classes of $Y(A, 0)$ are classified by the operations in Definition \ref{def:fund} (actually, they use the terms of zonotope and they prove more general theorem). 

\begin{theorem}{\rm (The unimodular and affine case of \cite[Corollary 5.4]{AP})}\\
For rank $d$ unimodular matrices $A$ and $A'$ in $\Mat_{d\times n}(\Z)$, the following are equivalent: 
\begin{itemize}
\item[(i)]$Y(A, 0)$ is $\Cstar\times\T_\C^{n-d}$-equivariant isomorphic to $Y(A', 0)$. 
\item[(ii)]$A\sim A'$.
\end{itemize}
{Moreover, in this case, $Y(A, 0)$ is $\Cstar\times\T_\C^{n-d}$-equivariant isomorphic to $Y(A', 0)$ as symplectic varieties.}
\end{theorem}  
As we will prove below, by using the description of the universal Poisson deformation space of $Y(A, 0)$ (cf.\ Theorem \ref{thm:Mainthm}), we can drop the condition on  $\T_\C^{n-d}$-equivariance in the theorem above.  Moreover, as noted in section 2, $A\sim A'$ if and only if the associated vector matroids $M(A)$ and $M(A')$ are isomorphic to each other. Then we can prove the following: 
\begin{theorem}\label{thm:classification}
For rank $d$ unimodular matrices $A$ and $A'$ in $\Mat_{d\times n}(\Z)$, the following are equivalent: 
\begin{itemize}
\item[(i)]$Y(A, 0)$ is isomorphic to $Y(A', 0)$ as conical symplectic varieties (cf.\ Definition \ref{def:conical}). 
\item[(ii)]$A\sim A'$, i.e., $A'$ is obtained from $A$ by a sequence of some elementary row operations over $\Z$, interchanging some column vectors $\bm{a_i}$, and multiplying some $\bm{a_i}$ by $-1$. 
\item[(iii)]The associated vector matroid $M(A)$ is isomorphic to $M(A')$.
\item[(iv)]The associated vector matroid $M(B^T)$ is isomorphic to $M({B'}^T)$, where $B$ (resp. $B'$) is a matrix which makes the following be exact: \[\setlength{\abovedisplayskip}{-3pt} \begin{tikzcd}
0\ar[r]&\Z^{n-d}\ar[r, "{B^{(')}}"] &\Z^{n}\ar[r, "{A^{(')}}"]&\Z^d\ar[r]&0         
\end{tikzcd}.\setlength{\belowdisplayskip}{-3pt} \]
\end{itemize}
\end{theorem}
\begin{proof}
The equivalence of (iii) and (iv) is clear since $M(B^T)$ is the dual matroid of $M(A)$ (cf.\ Remark \ref{rem:dual matroid}).
Then, by the theorem above, we only have to show that if $\phi :Y(A, 0)\xrightarrow{\sim}Y(A', 0)$ is an isomorphism as conical symplectic varieties, then { we can construct a (possibly another) $\Cstar\times\T_\C^{n-d}$-equivariant isomorphism $Y(A, 0)\xrightarrow{\sim}Y(A', 0)$ as symplectic varieties.}
Since $\phi :Y(A, 0)\xrightarrow{\sim}Y(A', 0)$ is an isomorphism as conical symplectic varieties, the corresponding universal Poisson deformation spaces are $\Cstar$-equivariant isomorphic to each other as Poisson varieties, moreover we have the following commutative diagram: 
\[\setlength{\abovedisplayskip}{-1pt} \begin{tikzcd}
X(A, 0)/W_B\ar[r, "\Phi", "\sim"']\ar[d]&X(A', 0)/W_{B'}\ar[d]\\
\C^d/W_B\ar[r, "\sim"']&\C^d/W_{B'}
\end{tikzcd}.\]
Additionally, by taking the fiber product of the above diagram with the natural projection $p: \C^d\to \C^d/W_B\cong\C^d/W_{B'}$, we have the following commutative diagram: 
\[\begin{tikzcd}[row sep=tiny, column sep=tiny]
X(A, 0)\ar[rr, "{\Phi}", "\sim"']\ar[dr]&&X(A', 0)\ar[dl]\\
&\C^d&
\end{tikzcd},\]
{where we also denote by $\Phi$ the induced isomorphism.} 
Note that {${\Phi}$} is a $\Cstar$-equivariant isomorphism as Poisson varieties between two (Poisson) toric varieties. {Then, we only have to show the below claim, where $\T_\C^{2n-d}:=\T_\C^{2n}/\T_\C^d$. In fact, if we can prove this, then the restriction of $\Psi$ to $Y(A, 0)$ will induce an $\Cstar\times\T_\C^{n-d}$-equivariant isomorphism $Y(A, 0) \to Y(A', 0)$ as symplectic varieties (cf.\ Lemma \ref{lem:fund}). }

{\underline{Claim}: One can replace $\Phi$ by a $\Cstar\times\T_\C^{2n-d}$-equivariant isomorphism ${\Psi}$ as Poisson varieties, which is in the form of an isomorphism as given in Lemma \ref{lem:fund}.}

{\underline{Proof}: First, by \cite[Theorem 4.1]{Ber}, we can replace ${\Phi}$ by a $\T_\C^{2n-d}$-equivariant isomorphism $\Psi_{\sigma} : X(A, 0) \xrightarrow{\sim} X(A', 0)$. More precisely, $\Psi_{\sigma}$ is induced from a permutation of coordinates $\widetilde{\Psi}_{\sigma} :\C^{2n} \xrightarrow{\sim} \C^{2n}$ ($\sigma\in\gerS_{2n}$) such that $\widetilde{\Psi}_{\sigma}$ is equivariant with respect to an isomorphism $\psi : \T_\C^d \xrightarrow{\sim} \T_\C^d$. In particular, by definition, $\widetilde{\Psi}_{\sigma}|_{\T_\C^{2n}}$ induces an isomorphism $\T_\C^{2n} \xrightarrow{\sim} \T_\C^{2n}$. Then, by the equivariance of $\widetilde{\Psi}_{\sigma}$ with respect to $\psi$, the following diagrams of tori and the corresponding character group commute:}
\[\begin{tikzcd}[ampersand replacement=\&]
\T_\C^d\ar[r, hook, "{\begin{pmatrix}A^T\\ \hline -A^T\end{pmatrix}}"]\ar[d, "\psi"', "\wr"]\&\T_\C^{2n}\ar[d, "\widetilde{\Psi}_{\sigma}"', "\wr"]\\
\T_\C^d\ar[r, hook, "{\begin{pmatrix}{A'}^T\\ \hline -{A'}^T\end{pmatrix}}"']\&\T_\C^{2n}
\end{tikzcd}, \ \ \ 
\begin{tikzcd}[ampersand replacement=\&]
\Z^d\&\Z^{2n}\ar[l, two heads, "{(A | -A)}"']\\
\Z^d\ar[u, "\psi^*", "\wr"']\&\Z^{2n}\ar[l, two heads, "{(A' | -A')}"]\ar[u, "\widetilde{\Psi}_{\sigma}^*", "\wr"']
\end{tikzcd}. 
\]
{Below, we want to construct a signed permutation matrix $D : \Z^n \to \Z^n$ such that $\widetilde{\Psi}^*:=(D, -D) : \Z^n\oplus\Z^n \to \Z^n\oplus \Z^n$ satisfying the above commutative diagram instead of $\widetilde{\Psi}^*_{\sigma}$. If do so, then we have $A=\psi^* A' D^{-1}$, and this induces a desired isomorphism by Lemma \ref{lem:fund}.} 

{To do this, by the commmutativity of the above diagram, note that we have}
\[\psi^*\bm{a'_i}=\begin{cases}
\bm{a_{\sigma(i)}}& (\text{if} \ 1\leq \sigma(i) \leq n)\\
-\bm{a_{\sigma(i)-n}}& (\text{if} \ n+1\leq \sigma(i) \leq 2n)
\end{cases}.\]
{Now, by multiplying some $\bm{a'_i}$'s by $-1$, we can assume that if $\bm{a'_i}\neq0$ then there is no $j$ such that $\bm{a'_i}=-\bm{a'_j}$ (cf.\ Lemma \ref{lem:fund}). In this setting, we will construct a permutation $\tau\in\gerS_n$ such that $\psi^*\bm{a'_i}=\pm\bm{a_i}$ for any $1\leq i\leq n$. First, take any bijection $\tau_0$ between $I'_0:=\{i\in [n] \ | \ \bm{a'_i}=\bm{0}\}$ and $I_0:=\{i\in [n] \ | \ \bm{a_i}=\bm{0}\}$, then define $\tau(i):=\tau_{0}(i)$. Next, for $i\notin I'_0$, we define $\tau(i)$ as }
\[\tau(i):=\begin{cases}
\sigma(i) & (\text{if} \ 1\leq \sigma(i) \leq n)\\
\sigma(i)-n &  (\text{if} \ n+1\leq \sigma(i) \leq 2n)
\end{cases}.\]
{Then, this defines a well-defined permutation $\tau\in\gerS_n$ since we assume that if $\bm{a'_i}\neq0$ then there is no $j$ such that $\bm{a'_i}=-\bm{a'_j}$. By definition, we have} 
\[\psi^*\bm{a'_i}=\begin{cases}
\bm{a_{\tau(i)}}& (\text{if} \ 1\leq \sigma(i) \leq n)\\
-\bm{a_{\tau(i)}}& (\text{if} \ n+1\leq \sigma(i) \leq 2n)
\end{cases}.\]
{Thus, $D:=D_\tau D_{\pm}$ is the desired signed permutation matrix, where $D_\tau$ is the permutation matrix associated to $\tau$, and $D_{\pm}:=\diag[\vep_1, \ldots, \vep_n]$ is the diagonal matrix with $\vep_i=1$ (if $1\leq \sigma(i) \leq n$) and $\vep_i=-1$ (if $n+1\leq \sigma(i) \leq 2n$). }



\end{proof}

{\begin{remark}
By the same argument, we can show that if {$\phi : Y(A, \alpha) \xrightarrow{\sim} Y({A'}, \alpha')$} is a $\Cstar$-equivariant isomorphism as symplectic varieties between two smooth hypertoric varieties, then $\phi$ {can be replaced by a $\Cstar\times\T_\C^{n-d}$-equivariant isomorphism as symplectic varieties. By \cite{AP}, two smooth hypertoric varieties $Y(A, \alpha)$ and $Y(A', \alpha')$ are $\Cstar\times\T_\C^{n-d}$-equivariant isomorphic if and only if the associated hyperplane arrangements $\calH_B$ and $\calH_{B'}$ define the same {\it zonotope tiling}.}

\end{remark}}

By applying the Whitney's 2-isomorphism theorem (Theorem \ref{thm:Whitney}), we get the criterion to determine whether two affine toric quiver varieties (cf.\ Example \ref{eg:quiver}) are isomorphic to each other or not in terms of the associated graphs.    
\begin{corollary}\label{cor:quiver}
Let $G_1$ and $G_2$ be two finite graphs (without isolated vertices). Then, the following are equivalent: 
\begin{itemize}
\item[(1)]$Y(A_{G_1}, 0)\cong Y(A_{G_2}, 0)$ as conical symplectic varieties (cf.\ Definition \ref{def:conical}).
\item[(2)]$G_2$ can be transformed into $G_1$ by a sequence of the following two operations (we describe these operations by examples below, for the more precise definition, see \cite[section 5.3]{Ox})
\begin{itemize}
\item[(i)]The vertex identification and the vertex cleaving.
\item[(ii)]The Whitney twist
\end{itemize}
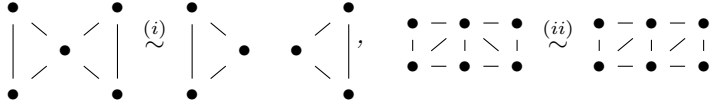
\begin{minipage}{0.3\hsize}
\[\begin{tikzcd}[row sep=tiny, column sep=tiny]
\bullet\ar[dd, no head]\ar[dr, no head]&&\bullet\ar[dd, no head]\ar[dl, no head]\\
&\bullet&\\
\bullet\ar[ur, no head]&&\bullet\ar[ul, no head]
\end{tikzcd}\overset{(i)}{\sim}
\begin{tikzcd}[row sep=tiny, column sep=tiny]
\bullet\ar[dd, no head]\ar[dr, no head]&&&\bullet\ar[dd, no head]\ar[dl, no head]\\
&\bullet&\bullet&\\
\bullet\ar[ur, no head]&&&\bullet\ar[ul, no head]
\end{tikzcd}\]
\end{minipage}
\begin{minipage}{0.1\hsize}
\ \ \ \ \ \ \ {,} 
\end{minipage}
\begin{minipage}{0.3\hsize}
\[\begin{tikzcd}[row sep=tiny, column sep=tiny]
\bullet\ar[d, no head]\ar[r, no head]&\bullet\ar[dl, no head]\ar[dr, no head]\ar[r, no head]\ar[d, no head]&\bullet\ar[d, no head]\\
\bullet\ar[r, no head]&\bullet\ar[r, no head]&\bullet\\
\end{tikzcd}\overset{(ii)}{\sim}
\begin{tikzcd}[row sep=tiny, column sep=tiny]
\bullet\ar[d, no head]\ar[r, no head]&\bullet\ar[dl, no head]\ar[r, no head]\ar[d, no head]&\bullet\ar[d, no head]\ar[dl, no head]\\
\bullet\ar[r, no head]&\bullet\ar[r, no head]&\bullet\\
\end{tikzcd}.
\]
\end{minipage}
\end{itemize}  
\end{corollary}

Now, using Theorem \ref{thm:classification}, we will describe all isomorphism classes of 4- and 6-dimensional affine hypertoric varieties (in this case, all ones are described as toric quiver varieties). Since $\dim Y(A, 0)=2(n-d)$, we only have to classify unimodular $n\times(n-d)$-matrices $B$ up to the equivalence relation when $n-d=2$ and 3. 

To describe concisely, we always write $B$ as the following form: 
\begin{equation*}
\arraycolsep=4pt
B=\left(
\begin{array}{ccc}
&\text{\huge{$B^{(1)}$}}&\\
\hline
&\text{\huge{$B^{(2)}$}}&\\
\hline
&\vdots&\\
\hline
&\text{\huge{$B^{(s)}$}}&\\
\end{array}
\right), \ \ \ 
B^{(k)}=\begin{array}{rcccll}
	\ldelim({3}{4pt}[] &&\bm{b^{(k)}}&&\rdelim){3}{4pt}[]&\rdelim\}{3}{10pt}[$\ell_k$] \\
	&&\vdots&& \\
	&& \bm{b^{(k)}} && \\
	\end{array} \ \ , \arraycolsep=2pt
\end{equation*}
where if $k_1 \neq k_2$, then $\bm{b^{(k_1)}}\neq\pm\bm{b^{(k_2)}}$. 
Also, we define the reduction $\overline{B}$ of $B$ as the following $s\times (n-d)$ matrix: 
\[\overline{B}:=\left(\begin{array}{ccc}
&\bm{b^{(1)}}&\\
&\bm{b^{(2)}}&\\
&\vdots&\\
&\bm{b^{(s)}}&
\end{array}\right).\]
In this case, we call $B=(\overline{B}: \ell_1, \ldots, \ell_s)$ the {\it reduced expression} of $B$. 
Since one can easily see the corresponding matroid $M(\overline{B}^T)$ to the reduction $\overline{B}^T$ is so-called the {\it simplification} of $M(B^T)$ (cf.\ \cite[p.49]{Ox}). Since the simplification of a matroid is unique up to matroid isomorphisms, we have the following.  
\begin{corollary}
Let $B$ (resp. $B'$) be a unimodular $n\times (n-d)$ matrix, and consider its simplification $\overline{B}\in\Mat_{d\times s}(\Z)$ (resp. $\overline{B'}\in\Mat_{d\times s'}(\Z)$). Take {$\underline{A}$ (resp. $\underline{A'}$)} as a unimodular matrix which satisfies the exact sequence $\begin{tikzcd}[column sep=1ex]0\ar[r]&[1ex]\Z^{n-s^{(')}}\ar[r, "{\overline{B}^{(')}}"] &[3ex]\Z^n\ar[r, "{{\underline{A}^{(')}}}"]&[3ex]\Z^{s^{(')}}\ar[r]&[1ex]0\end{tikzcd}$ respectively.  If $Y(A, 0)\cong Y(A', 0)$, then {$Y(\underline{A}, 0)\cong Y(\underline{A'}, 0)$}
\end{corollary}
Thus, we only have to classify regular matroids associated to totally unimodular matrices which don't have any parallel row vectors. Then, one can classify $\overline{B}$ in the case of $n-d=2$ or 3 by direct computations as the proposition below. All regular matroids of $\rank=2$ or 3 can be realized as graphic matroids $M(H)$ associated to some planar graphs $H$. Moreover, the simplification of $M(H)$ is the same as $M(\overline{H})$, where $\overline{H}$ is a graph obtained by replacing all multiple edges of $H$ by a simple edge. 
\begin{proposition}\hspace{2pt}\label{prop:classify_graph}\\
\begin{itemize}
\vspace{-5mm}
\item[(1)]If $n-d=2$, then the reduction $\overline{B}$ of $B$ will be equivalent to exactly one of the following and the associated matroid $M(\overline{B}^T)$ to each of them is the graphic matroid associated to the following graphs respectively.   
\[\overline{B}_1:=\begin{pmatrix}
1&0\\
0&1
\end{pmatrix}, \ \ \ M({\overline{B}_1}^T)\cong M(\overline{H}_1), \ \ \ \overline{H}_1=\begin{tikzcd}[row sep=small, column sep=small]\bullet\ar[r, no head]&\bullet\ar[r, no head]&\bullet\end{tikzcd},\]
\[  
\overline{B}_2:=\begin{pmatrix}
1&0\\
0&1\\
1&1
\end{pmatrix}, \ \ \ M({\overline{B}_2}^T)\cong M(\overline{H}_2), \ \ \ \overline{H}_2=\begin{tikzcd}[row sep=small, column sep=small]&\bullet&\\\bullet\ar[rr, no head]\ar[ur, no head]&&\bullet\ar[ul, no head]\end{tikzcd}.\]
\item[(2)]If $n-d=3$, then the reduction of $B$ will be equivalent to exactly one of the following: 
\[\overline{B}_3:=\begin{pmatrix}
1&0&0\\
0&1&0\\
0&0&1\\
\end{pmatrix},  \ \ \ M({\overline{B}_3}^T)\cong M(\overline{H}_3), \ \ \ \overline{H}_3=\begin{tikzcd}[row sep=small, column sep=small]
\bullet\ar[r ,no head]&\bullet\ar[r ,no head]&\bullet\ar[r ,no head]&\bullet
\end{tikzcd},\]
\[\overline{B}_4:=\begin{pmatrix}
1&0&0\\
0&1&0\\
0&0&1\\
1&1&0\\
\end{pmatrix},  \ \ \ M({\overline{B}_4}^T)\cong M(\overline{H}_4),  \ \ \ \overline{H}_4=\begin{tikzcd}[row sep=small, column sep=small]
&&\bullet\ar[dd ,no head]\\
\bullet\ar[r ,no head]&\bullet\ar[ur ,no head]\ar[dr ,no head]&\\
&&\bullet
\end{tikzcd},\]
\[\overline{B}_5:=\begin{pmatrix}
1&0&0\\
0&1&0\\
0&0&1\\
1&1&1
\end{pmatrix},  \ \ \ M({\overline{B}_5}^T)\cong M(\overline{H}_5),  \ \ \ \overline{H}_5=\begin{tikzcd}[row sep=small, column sep=small]
\bullet\ar[r ,no head]\ar[d ,no head]&\bullet\ar[d ,no head]\\
\bullet\ar[r ,no head]&\bullet
\end{tikzcd},\]
\[\overline{B}_6:=\begin{pmatrix}
1&0&0\\
0&1&0\\
0&0&1\\
1&1&0\\
1&0&1
\end{pmatrix}, \ \ \ M({\overline{B}_6}^T)\cong M(\overline{H}_6), \ \ \ \overline{H}_6=\begin{tikzcd}[row sep=small, column sep=small]
\bullet\ar[r ,no head]\ar[d ,no head]\ar[dr ,no head]&\bullet\ar[d ,no head]\\
\bullet\ar[r ,no head]&\bullet
\end{tikzcd},\]
\[\overline{B}_7:=\begin{pmatrix}
1&0&0\\
0&1&0\\
0&0&1\\
1&1&0\\
1&0&1\\
1&1&1
\end{pmatrix}, \ \ \ M({\overline{B}_7}^T)\cong M(\overline{H}_7), \ \ \ \overline{H}_7=\begin{tikzcd}[row sep=small, column sep=small]
&\bullet\ar[d ,no head]\ar[ddl ,no head]\ar[ddr ,no head]&\\
&\bullet\ar[dl ,no head]\ar[dr ,no head]&\\
\bullet\ar[rr ,no head]&&\bullet
\end{tikzcd}.\]

\end{itemize}
\end{proposition}
\begin{remark}\hspace{2pt}\label{rem:quiver2}\\
\begin{itemize}
\vspace{-5mm}
\item[(1)]
As remarked at Remark \ref{rem:graphic} (2), if $M(B^T)$ is isomorphic to a graphic matroid $M(H)$ corresponding to a planar graph $H$, then the dual matroid $M(A)$ is also isomorphic to a graphic matroid corresponding to some planar graph $G$. Since this means $A\sim A_G$, the corresponding hypertoric variety is the toric quiver variety associated to $G$ (cf.\ Example \ref{eg:quiver}). More explicitly, in 4- and 6-dimensional case, we can describe a double (framed) quiver $\overline{Q}_{G_i}^{\bm{\ell}}$ constructed by the process in Example \ref{eg:quiver} from $G_i^{\bm{\ell}}$ corresponding to $B_i$ whose reduced expression is $(\overline{B_i}: \bm{\ell}):=(\overline{B_i}: \ell_1, \ldots ,\ell_{s_i})$ as the following: 

\[\overline{Q}_{G_1^{{\ell_1, \ell_2}}}=\left(\begin{tikzcd}[row sep=small, column sep=small]
\circ\ar[r, shift left, ""]&\bullet\ar[r, shift left, ""]\ar[l, shift left, ""]&{}\ar[r, dashed, no head]\ar[l, shift left, ""]&{}\ar[r, shift left, ""]&\bullet\ar[r, shift left, ""]\ar[l, shift left, ""]&\circ\ar[l, shift left, ""]\\
\circ\ar[r, shift left, ""]&\bullet\ar[r, shift left, ""]\ar[l, shift left, ""]&{}\ar[r, dashed, no head]\ar[l, shift left, ""]&{}\ar[r, shift left, ""]&\bullet\ar[r, shift left, ""]\ar[l, shift left, ""]&\circ\ar[l, shift left, ""]\\
\end{tikzcd}\right),
\]
\[\overline{Q}_{G_2^{{\ell_1, \ell_2, \ell_3}}}=\left(\begin{tikzcd}[row sep=small, column sep=small]
&\bullet\ar[dl, shift left, ""]\ar[r, shift left, ""]&{}\ar[r, dashed, no head]\ar[l, shift left, ""]&{}\ar[r, shift left, ""]&\bullet\ar[l, shift left, ""]\ar[r, shift left, ""]&\circ\ar[l, shift left, ""]\\
\bullet\ar[r, shift left, ""]\ar[ur, shift left, ""]\ar[dr, shift left, ""]&\bullet\ar[l, shift left, ""]\ar[r, shift left, ""]&{}\ar[r, dashed, no head]\ar[l, shift left, ""]&{}\ar[r, shift left, ""]&\bullet\ar[l, shift left, ""]\ar[r, shift left, ""]&\circ\ar[l, shift left, ""]\\
&\bullet\ar[ul, shift left, ""]\ar[r, shift left, ""]&{}\ar[r, dashed, no head]\ar[l, shift left, ""]&{}\ar[r, shift left, ""]&\bullet\ar[l, shift left, ""]\ar[r, shift left, ""]&\circ\ar[l, shift left, ""]\\
\end{tikzcd}\right),
\]
\[\overline{Q}_{G_3^{{\ell_1, \ell_2, \ell_3}}}=\left(\begin{tikzcd}[row sep=small, column sep=small]
\circ\ar[r, shift left, ""]&\bullet\ar[r, shift left, ""]\ar[l, shift left, ""]&{}\ar[r, dashed, no head]\ar[l, shift left, ""]&{}\ar[r, shift left, ""]&\bullet\ar[r, shift left, ""]\ar[l, shift left, ""]&\circ\ar[l, shift left, ""]\\
\circ\ar[r, shift left, ""]&\bullet\ar[r, shift left, ""]\ar[l, shift left, ""]&{}\ar[r, dashed, no head]\ar[l, shift left, ""]&{}\ar[r, shift left, ""]&\bullet\ar[r, shift left, ""]\ar[l, shift left, ""]&\circ\ar[l, shift left, ""]\\
\circ\ar[r, shift left, ""]&\bullet\ar[r, shift left, ""]\ar[l, shift left, ""]&{}\ar[r, dashed, no head]\ar[l, shift left, ""]&{}\ar[r, shift left, ""]&\bullet\ar[r, shift left, ""]\ar[l, shift left, ""]&\circ\ar[l, shift left, ""]\\
\end{tikzcd}\right),
\]
\[\overline{Q}_{G_4^{{\ell_1, \ell_2, \ell_3, \ell_4}}}=\left(\begin{tikzcd}[row sep=small, column sep=small]
\circ\ar[r, shift left, ""]&\bullet\ar[r, shift left, ""]\ar[l, shift left, ""]&{}\ar[r, dashed, no head]\ar[l, shift left, ""]&{}\ar[r, shift left, ""]&\bullet\ar[r, shift left, ""]\ar[l, shift left, ""]&\circ\ar[l, shift left, ""]\\
&\bullet\ar[dl, shift left, ""]\ar[r, shift left, ""]&{}\ar[r, dashed, no head]\ar[l, shift left, ""]&{}\ar[r, shift left, ""]&\bullet\ar[l, shift left, ""]\ar[r, shift left, ""]&\circ\ar[l, shift left, ""]\\
\bullet\ar[r, shift left, ""]\ar[ur, shift left, ""]\ar[dr, shift left, ""]&\bullet\ar[l, shift left, ""]\ar[r, shift left, ""]&{}\ar[r, dashed, no head]\ar[l, shift left, ""]&{}\ar[r, shift left, ""]&\bullet\ar[l, shift left, ""]\ar[r, shift left, ""]&\circ\ar[l, shift left, ""]\\
&\bullet\ar[ul, shift left, ""]\ar[r, shift left, ""]&{}\ar[r, dashed, no head]\ar[l, shift left, ""]&{}\ar[r, shift left, ""]&\bullet\ar[l, shift left, ""]\ar[r, shift left, ""]&\circ\ar[l, shift left, ""]\\
\end{tikzcd}\right),\]
\[\overline{Q}_{G_5^{{\ell_1, \ell_2, \ell_3, \ell_4}}}=\left(\begin{tikzcd}[row sep=small, column sep=small]
\circ\ar[r, shift left, ""]&\bullet\ar[r, shift left, ""]\ar[l, shift left, ""]&{}\ar[r, dashed, no head]\ar[l, shift left, ""]&{}\ar[r, shift left, ""]&\bullet\ar[dr, shift left, ""]\ar[l, shift left, ""]&&\bullet\ar[dl, shift left, ""]\ar[r, shift left, ""]&{}\ar[r, dashed, no head]\ar[l, shift left, ""]&{}\ar[r, shift left, ""]&\bullet\ar[r, shift left, ""]\ar[l, shift left, ""]&\circ\ar[l, shift left, ""]\\
&&&&&\bullet\ar[dl, shift left, ""]\ar[ur, shift left, ""]\ar[ul, shift left, ""]\ar[dr, shift left, ""]&&&&&\\
\circ\ar[r, shift left, ""]&\bullet\ar[r, shift left, ""]\ar[l, shift left, ""]&{}\ar[r, dashed, no head]\ar[l, shift left, ""]&{}\ar[r, shift left, ""]&\bullet\ar[ur, shift left, ""]\ar[l, shift left, ""]&&\ar[ul, shift left, ""]\bullet\ar[r, shift left, ""]&{}\ar[r, dashed, no head]\ar[l, shift left, ""]&{}\ar[r, shift left, ""]&\bullet\ar[r, shift left, ""]\ar[l, shift left, ""]&\circ\ar[l, shift left, ""]\\
\end{tikzcd}\right),
\]
\[\overline{Q}_{G_6^{{\ell_1, \ell_2, \ell_3, \ell_4, \ell_5}}}=\left(\begin{tikzcd}[row sep=small, column sep=small]
\circ\ar[r, shift left, ""]&\bullet\ar[l, shift left, ""]\ar[r, shift left, ""]&{}\ar[r, dashed, no head]\ar[l, shift left, ""]&{}\ar[r, shift left, ""]&\bullet\ar[l, shift left, ""]\ar[r, shift left, ""]&\bullet\ar[l, shift left, ""]\ar[r, shift left, ""]\ar[d, shift left, ""]&\bullet\ar[r, shift left, ""]\ar[l, shift left, ""]&{}\ar[r, dashed, no head]\ar[l, shift left, ""]&{}\ar[r, shift left, ""]&\bullet\ar[r, shift left, ""]\ar[l, shift left, ""]&\circ\ar[l, shift left, ""]\\
&&&&&\bullet\ar[d, shift left, ""]\ar[u, shift left, ""]&&&&&\\
&&&&&{}\ar[u, shift left, ""]\ar[d, dashed, no head]&&&&&\\
&&&&&{}\ar[d, shift left, ""]&&&&&\\
\circ\ar[r, shift left, ""]&\bullet\ar[l, shift left, ""]\ar[r, shift left, ""]&{}\ar[r, dashed, no head]\ar[l, shift left, ""]&{}\ar[r, shift left, ""]&\bullet\ar[l, shift left, ""]\ar[r, shift left, ""]&\bullet\ar[l, shift left, ""]\ar[r, shift left, ""]\ar[u, shift left, ""]&\bullet\ar[r, shift left, ""]\ar[l, shift left, ""]&{}\ar[l, shift left, ""]\ar[r, dashed, no head]&{}\ar[r, shift left, ""]&\bullet\ar[r, shift left, ""]\ar[l, shift left, ""]&\circ\ar[l, shift left, ""]
\end{tikzcd}\right),
\]
\[\overline{Q}_{G_7^{{\ell_1, \ell_2, \ell_3, \ell_4, \ell_5, \ell_6}}}=\left(\begin{tikzcd}[row sep=small, column sep=small]
&&&&&&&\bullet\ar[dl, shift left, ""]\ar[r, shift left, ""]\ar[d, shift left, ""]&{}\ar[r, dashed, no head]\ar[l, shift left, ""]&{}\ar[r, shift left, ""]&\bullet\ar[r, shift left, ""]\ar[l, shift left, ""]&\circ\ar[l, shift left, ""]\\
&&&&&&{}\ar[ur, shift left, ""]&\bullet\ar[d, shift left, ""]\ar[u, shift left, ""]&&&&\\
&&&&&{}\ar[dl, shift left, ""]\ar[ur, dashed, no head]&&{}\ar[u, shift left, ""]\ar[d, dashed, no head]&&&&\\
\circ\ar[r, shift left, ""]&\bullet\ar[r, shift left, ""]\ar[l, shift left, ""]&{}\ar[r, dashed, no head]\ar[l, shift left, ""]&{}\ar[r, shift left, ""]&\bullet\ar[l, shift left, ""]\ar[dr, shift left, ""]\ar[ur, shift left, ""]&&&{}\ar[d, dashed, no head]&&&&\\
&&&&&{}\ar[ul, shift left, ""]\ar[dr, dashed, no head]&&{}\ar[d, shift left, ""]&&&&\\
&&&&&&{}\ar[dr, shift left, ""]&\bullet\ar[d, shift left, ""]\ar[u, shift left, ""]&&&&\\
&&&&&&&\bullet\ar[ul, shift left, ""]\ar[r, shift left, ""]\ar[u, shift left, ""]&{}\ar[r, dashed, no head]\ar[l, shift left, ""]&{}\ar[r, shift left, ""]&\bullet\ar[r, shift left, ""]\ar[l, shift left, ""]&\circ\ar[l, shift left, ""]
\end{tikzcd}\right),\]
where the length of each chain is equal to each multiplicity $\ell_1, \ldots ,\ell_s$ respectively. 
\item[(2)]As noted at Remark \ref{rem:graphic} (2), if $M(B^T)$ is not isomorphic to any graphic matroid associated to a planar graph, then $M(A)$ cannot be realized as a graphic matroid. For example, if we take $H=K_5$, that is the complete graph on 5 vertices, then the corresponding 8-dimensional hypertoric variety is not isomorphic to any toric quiver variety as conical symplectic varieties.    
\end{itemize}   
\end{remark}

From the proposition above, we can describe more explicitly all non-isomorphic 4-dimensional affine hypertoric varieties as the following. 

\begin{theorem}\label{thm:4dimension}
Each 4-dimensional affine hypertoric variety $Y(A, 0)$ associated to a unimodular matrix $A$ is isomorphic to exactly one of the following as conical symplectic varieties:
\begin{itemize}
\item[(i)]$S_{A_{\ell_1-1}}\times S_{A_{\ell_2-1}}$.
\item[(ii)]$\overline{{\calO}^{\tmin}}(\{\ell_1, \ell_2, \ell_3\}):=\Set{\begin{pmatrix}u_1&x_{12}&x_{13}\\y_{12}&u_2&x_{23}\\y_{13}&y_{23}&u_3\end{pmatrix} \in \gersl_3 \ | \ \text{All 2-minors of} \begin{pmatrix}u_1^{\ell_1}&x_{12}&x_{13}\\y_{12}&u_2^{\ell_2}&x_{23}\\y_{13}&y_{23}&u_3^{\ell_3}\end{pmatrix}=0}$,
\end{itemize}
where $S_{A_{\ell-1}}$ is the $A_{\ell-1}$ type surface singularity. 
Furthermore, the Namikawa--Weyl group $W$ is given by the following respectively: 
\begin{itemize}
\setlength{\leftskip}{1.0cm}
\item[Case(i):]$W=\gerS_{\ell_1}\times\gerS_{\ell_2}$.
\item[Case(ii):]$W=\gerS_{\ell_1}\times\gerS_{\ell_2}\times\gerS_{\ell_3}$.
\end{itemize}
\end{theorem} 
\begin{proof}
By Proposition \ref{prop:classify_graph} (1), we only have to show that if the reduced expression of $B$ is $(\overline{B_1}: \ell_1, \ell_2)$ (resp. $(\overline{B_2}: \ell_1, \ell_2, \ell_3)$), then the corresponding affine hypertoric variety is isomorphic to $S_{A_{\ell_1-1}}\times S_{A_{\ell_2-1}}$ (resp. $\overline{{\calO}^{\tmin}}(\{\ell_1, \ell_2, \ell_3\})$). The first case is straightforward by combining Example \ref{Amsurf} and Remark \ref{rem:product}.  
On the second case, firstly we describe the case of $\ell_1=\ell_2=\ell_3=1$ (cf.\ Example \ref{Amnilpotent}). For our convenience, we take $B$ and $A$ as the following: 
\[B=\begin{pmatrix}1&0\\0&1\\-1&-1\end{pmatrix}, \ \ \ A=\begin{pmatrix}1&1&1\end{pmatrix}.\]  
Then, by Lemma \ref{ring str lem}, we can compute the coordinate rings of $X(A, 0)$ and $Y(A, 0)$ as the following:  
\begin{equation*}
\begin{split}
\C[X(A, 0)]&=\C[z_1w_1, z_2w_2, z_3w_3][z_1w_2, z_1w_3, z_2w_3, \ w_1z_2, w_1z_3, w_2z_3] \\
&\cong\C[u_1, u_2, u_3, x_{12}, x_{13}, x_{23}, y_{12}, y_{13}, y_{23}]/I, \\
\end{split}
\end{equation*}
\[\C[Y(A, 0)]\cong\C[u_1, u_2, u_3, x_{12}, x_{13}, x_{23}, y_{12}, y_{13}, y_{23}]/I+J_{\mu}, \]
\[\text{where} \ 
I=\left\langle \text{All $2\times2$-minors of} \begin{pmatrix}u_1&x_{12}&x_{13}\\y_{12}&u_2&x_{23}\\y_{13}&y_{23}&u_3\end{pmatrix} \right\rangle, \ \ \ \ J_\mu=\langle u_1+u_2+u_3\rangle.
\]
Thus, 
\[Y(A, 0) \cong \overline{{\calO}^{\tmin}}:=\Set{C \in \gersl_3\ | \ \text{All $2\times2$-minors of} \ C=0}.\]
In the case of general $\ell_1, \ell_2, \ell_3$, note that we can take $B$ and $A$ as the following: 
\[
B=\begin{array}{rccll}
	\ldelim({13}{4pt}[] &1&0&\rdelim){13}{4pt}[]& \\
	&0&1&&\\
\cline{2-3}
	&1&0&&\rdelim\}{3}{10pt}[$\ell_1-1$]\\
	&\vdots&\vdots& & \\
	&1&0& & \\
	&0&1& & \rdelim\}{3}{10pt}[$\ell_2-1$,] \\
	&\vdots&\vdots& & \\
	&0&1& & \\
	&-1 &-1 & & \rdelim\}{3}{10pt}[$\ell_3$]\\
	&\vdots&\vdots& & \\
	&-1&-1&& \\
	\end{array}  \ \ \ \ \ \ \ \ \ \ 
A=\left(\begin{array}{cc|ccc|ccc|ccc}
-1&0&&&&&&&&&\\
\vdots&\vdots&\vcbig{I_{\ell_1-1}}&&\vcbig{0}&&\vcbig{0}&\\
-1&0&&&&&&&&&\\
\cline{1-11}
0&-1&&&&&&&&&\\
\vdots&\vdots&\vcbig{0}&&\vcbig{I_{\ell_2-1}}&&\vcbig{0}&\\
0&-1&&&&&&&&&\\
\cline{1-11}
1&1&&&&&&&&&\\
\vdots&\vdots&\vcbig{0}&&\vcbig{0}&&\vcbig{I_{\ell_3}}&\\
1&1&&&&&&&&&
\end{array}\right).
\]
Then, one can directly compute the coordinate ring of $\overline{{\calO}^{\tmin}}(\{\ell_1,\ell_2,\ell_3\})$ by similar arguments, and one can get the desired description as the statement (ii). This completes the proof. 
\end{proof}
\begin{remark}\hspace{2pt}\\
\begin{itemize}
\vspace{-5mm}
\item[(1)]$\overline{{\calO}^{\tmin}}(\{1,1,1\})$ is the (usual) minimal nilpotent orbit closure $\overline{{\calO}^{\tmin}_{A_2}}$ of $A_2$ type. 
\item[(2)]By the similar construction above, one can consider the hypertoric varieties $\overline{{\calO}^{\tmin}}(\{\ell_1, \ldots, \ell_{m+1}\})$ as the following: 
\[\overline{{\calO}^{\tmin}}(\{\ell_1, \ldots, \ell_{m+1}\}):=\left\{C\in \gersl_{m+1} \ \Bigl|\Bigr. \ \text{All $2\times 2$-minors of} \ C(\ell_1, \ldots ,\ell_{m+1})=0\right\},
\]
where for $C$ such that its diagonal is $(u_1, \ldots ,u_{m+1})$, we define $C(\ell_1, \ell_2, \ldots, \ell_{m+1})$ as the matrix which has $(u_1^{\ell_1}, \ldots, u_{m+1}^{\ell_{m+1}})$ as the diagonal and other components are the same as $C$. 
Furthermore, in this case, the Namikawa--Weyl group $W$ is given by the following: 
\[W=\gerS_{\ell_1}\times\cdots\times\gerS_{\ell_{m+1}}.\]
\end{itemize}
\end{remark}
In the next section, we will count the number of all projective crepant resolutions of $\overline{{\calO}^{\tmin}}(\{\ell_1, \ell_2, \ell_3\})$.

\bigskip

In the remainder of this section, we will determine the necessary and sufficient condition for an affine hypertoric variety $Y(A, 0)$ to be $\Cstar$-equivariant isomorphic to a  nilpotent orbit closure $\overline{\calO}$ of some semisimple Lie algebra as an algebraic variety. Although this result is not a corollary of Theorem \ref{thm:classification}, we will prove this here. 

The more precise statement is the following. 
\begin{theorem}\label{thm:char_nilpotent}
{Let $\overline{\calO}$ be a nilpotent orbit closure of some semisimple Lie algebra. Then, there exists some (not necessarily unimodular) matrix $A$ such that $Y(A, 0)$ is $\Cstar$-equivariant isomorphic to $\overline{\calO}$ if and only if }
\[\overline{\calO}\cong\overline{{\calO}^{\tmin}_{A_{\ell_1}}}\times\cdots\times\overline{{\calO}^{\tmin}_{A_{\ell_s}}} ,\]
where $\overline{{\calO}^{\tmin}_{A_{\ell}}}$ is the minimal nilpotent orbit closure of type $A_{\ell}$ (\text{cf.\ Example \ref{Amnilpotent}}), {and we consider $\overline{\calO}$ as a conical symplectic variety with weight 2 under the square $\Cstar$-action of the usual $\Cstar$-action on $\overline{\calO}$ with weight 1.}   
\end{theorem}  

In general, for a conical symplectic variety $Y_0$, the $\Z_{\geq0}$-grading on the coordinate ring $\C[Y_0]$ of $Y_0$ is induced from the $\Cstar$-action. The {\it maximal weight} of $Y_0$ is defined as the maximal degree of minimal generators of $\C[Y_0]$ with respect to the grading. In the case of a nilpotent orbit closure $\overline{\calO}$ in some semisimple Lie algebra $\gerg$, the maximal weight is 1 (moreover, in \cite{Namcharacterize}, {it is shown that} this property characterizes nilpotent orbit closures among all conical symplectic varieties). {Since we consider $\overline{\calO}$ as a conical symplectic variety with $\Cstar$-action of weight 2  under the square of the usual action as above, this means that $\C[\overline{\calO}]$ is generated by degree two elements.} 
{On the other hand, as remarked at Remark \ref{rem:deg=1},} there exists elements of degree two in minimal generators of $\C[Y(A, 0)]$. Thus, if $\overline{\calO}\cong Y(A, 0)$ as conical symplectic varieites, then $\C[Y(A, 0)]$ will be generated by elements of degree two in $\C[Y(A, 0)]$. Since $\C[Y(A, 0)]$ is the quotient ring of $\C[X(A, 0)]$ by elements of degree two (Lemma \ref{ring str lem}), if $\C[Y(A, 0)]$ is generated by elements of degree two, then $\C[X(A, 0)]$ is so. 
By Lemma \ref{ring str lem}, we have $\C[X(A, 0)]=\C[z_1w_1, \ldots ,z_nw_n][f_{\bm{\beta}} \ | \ \bm{\beta}\in\Image B]$, where $f_{\bm{\beta}}:=\prod_{i: \beta_i>0}{z_i^{\beta_i}}\prod_{i: \beta_i<0}{w_i}^{-\beta_i}$. Thus,  elements of degree two in $\C[X(A, 0)]$ are $\{z_1w_1, \ldots ,z_nw_n\}\sqcup(\Image B)_2$, where $(\Image B)_2:=\{\bm{\beta}\in\Image B \ | \ \bm{\beta}=\pm\bm{e_i}\pm\bm{e_j} \ (1\leq i<j\leq n) \}$ (by Remark \ref{rem:deg=1}, $\pm2\bm{e_i}\notin \Image B \ (1\leq i\leq n)$). Hence, what we have to show is the following (cf.\ Example \ref{eg:product}).       
\begin{proposition}
If $\Image B$ is generated by $\bm{u_1}, \ldots ,\bm{u_N}$, where $(\Image B)_2:=\{\bm{u_1}, \ldots ,\bm{u_N}\}$, then $A$ is equivalent to a matrix as the following form: 
\[\begin{pmatrix}
  \smash[b]{\block{\ell_1+1}}&                              &          &                     \\
                              & \smash[b]{\block{\ell_2+1}} &          &  \text{\huge{0}}\\
                              &                              & \ddots &                     \\
   \text{\huge{0}}      &                              &           & \block{\ell_s+1}
  \end{pmatrix},\]
where $\ell_1, \ldots ,\ell_s\geq0$.
In particular, $Y(A, 0)\cong\overline{{\calO}^{\tmin}_{A_{\ell_1}}}\times\cdots\times\overline{{\calO}^{\tmin}_{A_{\ell_s}}}$, where we define $\overline{\calO}_{A_{\ell}}^{\tmin}:=\{pt\}$ if $\ell=0$.
\end{proposition}
\begin{proof}
First, we want to show that there exists $1\leq i_1<\cdots <i_{n-d}\leq N$ such that $B\sim U_0$, where $U_0:=[\bm{u_{i_1}}, \ldots ,\bm{u_{i_{n-d}}}]$. Since $\Coker B=\Coker U\cong\Z^d$ is a free abelian group, so we have 
\[\gcd(\text{$(n-d)\times(n-d)$-minors of $B$})=\gcd(\text{$(n-d)\times(n-d)$-minors of $U$})=1.\]
Thus, we only have to show that there exists $1\leq i_1<\cdots <i_{n-d}\leq N$ such that $\gcd(\text{$(n-d)\times(n-d)$-minors of $U_0$})=1$. 

In general, by indcution, one can show that any minors of the matrix whose all column vectors are in the form of $\pm\bm{e_i}\pm\bm{e_j}$ is 0 or $2^\ell$ \ $(\ell\geq0)$. In particular, since $\gcd(\text{$(n-d)\times(n-d)$-minors of $U$})=1$, this implies that there exists $1\leq i_1<\cdots <i_{n-d}\leq N$ such that a minor of $U_0$ is 1. Thus, we have $\gcd($\text{$(n-d)\times(n-d)$-minors of $U_0$}$)=1$. 

Since each column vector of $U_0$ is in the form of $\pm\bm{e_i}\pm\bm{e_j}$, by the induction on the size of matrices, one can easily show that $U_0$ can be transformed to the following form: 
\[U_0\sim\left(\begin{array}{cc}
\vncbig{I_{n-d}}\\
\cline{1-2}
&\\
\vcbig{C}
\end{array}\right)=:B',\] 
where each column vector of $C$ is in the form of $\pm\bm{e_i}$. Then for this $B'$, one can take $A'$ as the following by the same reason as Remark \ref{rem:fund}. 
\[A'=\left(\begin{array}{cc|cc}
\vncbig{-C}&\vncbig{I_d}
\end{array}\right).
\]
Note that each column vector of $C$ is $\pm\bm{e_i}$ for some $1\leq i \leq d$. Thus, if we interchange column vectors and multiply by $\pm1$, then we will obtain
\[A'\sim A'':=\begin{pmatrix}
  \smash[b]{\block{\ell_1+1}}&                              &          &                     \\
                              & \smash[b]{\block{\ell_2+1}} &          &  \text{\huge{0}}\\
                              &                              & \ddots &                     \\
   \text{\huge{0}}      &                              &           & \block{\ell_s+1}
  \end{pmatrix},\]
where $\ell_1, \ldots ,\ell_s\geq0$. 
This completes the proof.  
\end{proof}
\section{Counting projective crepant resolutions of affine hypertoric varieties}
In this section, we will show that the number of all (projective) crepant resolutions of an affine hypertoric variety $Y(A, 0)$ can be computed from the hyperplane arrangement $\calH_A$ generated by some column vectors $\bm{a_i}$ of $A$.
In application, we will compute the number of crepant resolutions of $\overline{\calO^{\tmin}}(\{\ell_1, \ell_2, \ell_3\})$ explicitly in the case of $\ell_3=1, 2$. 

In general, the number of distinct projective crepant resolutions of conical symplectic varieties (or more generally, rational Gorenstein singularities) is equal to the number of ample cones inside the movable cone (cf.\ \cite{Ya}). Furthermore, in the case of conical symplectic varieties, the movable cone is a fundamental domain with respect to the Namikawa--Weyl group action on $H^2(\tilde{Y}, \R)=\Pic_\R(\tilde{Y})$ (\cite[Proposition 2.17]{BPW1}). In particular, the following holds.  
\begin{proposition}Assume a conical symplectic variety $Y$ admits a projective crepant resolution $\pi: \tilde{Y}\to Y$. Then the number of all distinct projective crepant resolutions of $Y$ is given by the following: 
\[\fr<\#(\text{the chambers of $\calH_Y$})/|W|>,\]
where $\calH_Y\subset H^2(\tilde{Y}, \R)=\Pic_\R(\pi)$ is the associated hyperplane arrangement whose each chamber is an ample cone of a projective crepant resolution of $Y$, and $W$ is the Namikawa--Weyl group for $Y$. 
\end{proposition}

For an affine hypertoric variety $Y(A, 0)$, we considered in section 2 the hyperplane arrangement $\calH_A$ in $\R^d$ defined by the following. 
\vspace{-3pt}

\[\calH_A:=\{H {\subset \R^d}\ | \ H \ \text{is a subspace generated by some $\bm{a_j}$'s and $\codim H=1$}\}.\]

By \cite[Theorem 6.10]{Kovariation}, the K\"{a}hler cone of $\pi: Y(A, \alpha)\to Y(A, 0)$ for each generic $\alpha\in\Z^d$ is equal to the chamber $C$ of \ $\calH_A$ which includes $\alpha$ through the Kirwan map $\kappa_2: \R^d\overset{\sim}{\to} H^2(Y(A, \alpha), \R)=\Pic_\R(\pi)$. Furthermore, as remarked at \cite[Lemma 4.14]{BK}, each line bundle in $C$ is $\pi$-ample. These imply that $C$ is the $\pi$-ample cone. 
Consequently, we have $\calH_{Y(A, 0)}=\calH_A$.
Using this fact and the description of the Namikawa--Weyl group of $Y(A, 0)$ (cf.\ Theorem \ref{thm:Mainthm}), we can compute the distinct projective crepant resolutions of $Y(A, 0)$ as the following. 
\begin{corollary}\label{cor:counting}
For an affine hypertoric variety $Y(A, 0)$ associated to a unimodular matrix $A$, Then the number of its all distinct projective crepant resolutions is given by the following: 
\[\fr<r(\calH_A)/|W_B|>,\]
where $r(\calH_A)$ is the number of chambers of $\calH_A$.
\end{corollary}
\begin{remark}Although we consider only projective crepant resolutions of $Y(A, 0)$ in this paper, it is known that some affine hypertoric varieties admit {\it non-projective} proper crepant resolutions (\cite{AP}). In \cite[Example 6.10]{AP}, an example of non-projective crepant resolutions of $\overline{\calO^{\tmin}}(\{3,3,3\})$ is described in terms of zonotopal tilings.  
\end{remark}
In general, to compute the number of chambers of a hyperplane arrangement $\calH$, we consider a more refined invariant, the characteristic polynomial $\chi_{\calH}(t)$ defined as below. 

Recall some notations on hyperplane arrangements. Given an arrangement $\calH$ in $\R^d$, let $L(\calH)$ be the set of all nonempty intersections of hyperplanes in $\calH\subset\R^d$. We define the partial order $x\leq y$ in $L(\calH)$ if $x\supseteq y$. We call $L(\calH)$ the {\it intersection poset of $\calH$}. 

For $L(\calH)$ (in general for any finite poset with the least element), we can define the {\it M\"{o}bius fuction} $\mu: L(\calH) \to \Z$ as
\[\mu(\R^d)=1 \ \ \text{and} \ \ \mu(x)=-\sum_{y<x}{\mu(y)}.\]
\begin{definition}
The {\it characteristic polynomial} $\chi_{\calH}(t)$ of the hyperplane arrangement $\calH$ is defined by
\[\chi_{\calH}(t):=\sum_{x\in L(\calH)}{\mu(x)t^{\dim(x)}},\]
where $\dim(x)$ is the dimension of $x$ as an affine subspace of $\R^d$.
\end{definition}
One of the important properties of the characteristic polynomial is the following.   
\begin{theorem}{\rm (\cite{Zas})}\label{thm:Zas}\\
For any hyperplane arrangement $\calH$ in $\R^d$, we have
\[\#\{\text{chamber of} \ \calH\}=(-1)^d\chi_{\calH}(-1).\]
\end{theorem}
From here, we mainly consider hyperplane arrangements which are defined over $\Z$ (i.e., all the coefficients of all hyperplanes in $\calH$ are integers). For such hyperplane arrangements, we can reduce the coefficients of each hyperplane $H\in\calH$ modulo $p$ to get a hyperplane $\overline{H} \subset \F_p^d$, where $p$ is a prime number. 
This reduction is enable us to compute the charasteristic polynomial of $\calH$ as the following. 
\begin{theorem}{\rm (The finite field method \cite{Ath})}\\
 Let $\calH\subset\R^d$ be a hyperplane arrangement defined over $\Z$. If $p$ is a sufficient large prime number, then
\[\chi_{\calH}(p)=\left|\F_p^d\setminus\bigcup_{H\in\calH}{\overline{H}}\right|.\]
\end{theorem}

We now consider the hyperplane arrangement associated to the hypertoric variety $\overline{\calO^{\tmin}}(\{\ell_1, \ell_2, \ell_3\})$. 
At first, we remark that $\overline{\calO^{\tmin}}(\{\ell_1, \ell_2, \ell_3\})$ is the hypertoric variety defined from the following matrices $A=A_{\ell_1, \ell_2, \ell_3}$ and $B=B_{\ell_1, \ell_2, \ell_3}$ (cf.\ see the proof of Theorem \ref{thm:4dimension}).
\[
B=\begin{array}{rccll}
	\ldelim({13}{4pt}[] &1&0&\rdelim){13}{4pt}[]& \\
	&0&1&&\\
\cline{2-3}
	&1&0&&\rdelim\}{3}{10pt}[$\ell_1-1$]\\
	&\vdots&\vdots& & \\
	&1&0& & \\
	&0&1& & \rdelim\}{3}{10pt}[$\ell_2-1$,] \\
	&\vdots&\vdots& & \\
	&0&1& & \\
	&-1 &-1 & & \rdelim\}{3}{10pt}[$\ell_3$]\\
	&\vdots&\vdots& & \\
	&-1&-1&& \\
	\end{array}  \ \ \ \ \ \ \ \ \ \ 
A=\left(\begin{array}{cc|ccc|ccc|ccc}
-1&0&&&&&&&&&\\
\vdots&\vdots&\vcbig{I_{\ell_1-1}}&&\vcbig{0}&&\vcbig{0}&\\
-1&0&&&&&&&&&\\
\cline{1-11}
0&-1&&&&&&&&&\\
\vdots&\vdots&\vcbig{0}&&\vcbig{I_{\ell_2-1}}&&\vcbig{0}&\\
0&-1&&&&&&&&&\\
\cline{1-11}
1&1&&&&&&&&&\\
\vdots&\vdots&\vcbig{0}&&\vcbig{0}&&\vcbig{I_{\ell_3}}&\\
1&1&&&&&&&&&
\end{array}\right).
\]
To describe the associated hyperplane arrangement $\calH_{A_{\ell_1, \ell_2, \ell_3}}\subset\R^{\ell_1+\ell_2+\ell_3-2}$, we remark a generalization of the Gale duality (cf.\ \cite[Lemma 14.3.1]{CLS}). 
\begin{lemma}{\rm (A generalization of the Gale duality)}\label{lem:Gale}\\
Assume $A=[\bm{a_1}, \ldots ,\bm{a_n}] \in M_{d \times n}(\R)$ and $B=[\bm{b_1}, \ldots ,\bm{b_n}]^T \in  M_{n \times (n-d)}(\R)$ satisfy the exact sequence 
\[\begin{tikzcd}0\ar[r]&\R^{n-d}\ar[r, "B"]&\R^n\ar[r, "A"]&\R^d\ar[r]&0\end{tikzcd}.\]
For a partition $\{1,2, \ldots ,n\}=I\sqcup J$ and $r \in \Z_{\geq 0}$ satisfying $0 \leq |I|-r \leq n-d, \ 0 \leq r \leq d$, we have 
\[
\dim_\R\Span_\R(\bm{b_i} \ | \ i \in I)=|I|-r \Leftrightarrow \dim_\R\Span_\R(\bm{a_j} \ | \ j \in J)=d-r
.\] 
\end{lemma}

\begin{proof}
$(\Rightarrow)$ Note \ $\Span_\R(\bm{a_j} \ | \ j \in J)=\Image(A|_{\R^J})$, where $\R^J:=\{(v_1, \ldots ,v_n)^T \in \R^n \ | \ v_i=0 \ (i \notin J) \}$.
Then, by the assumption, 
\begin{align*}
\dim\Ker(A|_{\R^J})=\dim(\Image B\cap \R^J)&=\dim\{\bm{h} \in \R^{n-d} \ | \ \langle \bm{b_i}, \bm{h}\rangle_{\R^{n-d}}=0 \ (i \in I)\}\\
&=n-d-\dim \Span_\R(\bm{b_i} \ | \ i \in I)\\
&=|J|-(d-r),
\end{align*}
where $\langle-, -\rangle:\R^{n-d}\times(\R^{n-d})^*\to\R$ is the standard pairing.
 
$(\Leftarrow)$ Set $q:=|J|-d+r$. Then, the following is clear: 
\[0 \leq |I|-r \leq n-d \ \text{and} \ 0 \leq r \leq d \ \ \ \Leftrightarrow \ \ \ 0 \leq |J|-q \leq d\ \text{and}\ 0 \leq q \leq n-d.
\]
Thus, one can show the claim by applying the above argument to the dual exact sequence $\begin{tikzcd}0\ar[r]&(\R^{d})^*\ar[r, "A^T"]&(\R^n)^*\ar[r, "B^T"]&(\R^{n-d})^*\ar[r]&0\end{tikzcd}$.
\end{proof}

Using this lemma in the case of $r=1$, one can easily describe $\calH_{A_{\ell_1, \ell_2, \ell_3}}\subset\R^{\ell_1+\ell_2+\ell_3-2}$ as the following, where we take $(x'_1, \ldots ,x'_{\ell_1-1}, y'_1, \ldots ,y'_{\ell_2-1}, z'_1, \ldots ,z'_{\ell_3})$ as the coordinates of $\R^{\ell_1+\ell_2+\ell_3-2}$.  
\[\calH_{A_{\ell_1, \ell_2, \ell_3}}=\Set{\begin{array}{lll}
H^x_i : &x'_i=0&(1\leq i\leq \ell_1-1)\\
H^y_j : &y'_j=0&(1\leq j\leq \ell_2-1)\\
H^z_k : &z'_k=0&(1\leq k\leq \ell_3)\\
H^{xz}_{ik} : &x'_i+z'_k=0&(1\leq i\leq \ell_1-1, \ 1\leq k\leq \ell_3)\\
H^{yz}_{jk} : &y'_j+z'_k=0&(1\leq j\leq \ell_2-1, \ 1\leq k\leq \ell_3)\\
H_{ijk} : &x'_i+y'_j+z'_k=0&(1\leq i\leq \ell_1-1, \ 1\leq j\leq \ell_2-1, \ 1\leq k\leq \ell_3)\\
H^x_{i_1,i_2} : &x'_{i_1}-x'_{i_2}=0&(1\leq i_1< i_2\leq \ell_1-1)\\
H^y_{j_1,j_2} : &y'_{j_1}-y'_{j_2}=0&(1\leq j_1< j_2\leq \ell_2-1)\\
H^z_{k_1,k_2} : &z'_{k_1}-z'_{k_2}=0&(1\leq k_1< k_2\leq \ell_3)
\end{array}}.\]
Computing the characteristic polynomial of $\calH_{A_{\ell_1, \ell_2, \ell_3}}$ is reduced to compute the one of another well-known hyperplane arrangement as the following. 
\begin{lemma}Define the hyperplane arrangement $\calA_{\ell_1, \ell_2, \ell_3}\subset\R^{\ell_1+\ell_2+\ell_3}$ as the following, where we take $(x_1, \ldots ,x_{\ell_1}, y_1, \ldots ,y_{\ell_2}, z_1, \ldots ,z_{\ell_3})$ as the coordinate of $\R^{\ell_1+\ell_2+\ell_3}$.
\[\calA_{\ell_1, \ell_2, \ell_3}:=\Set{\begin{array}{lll}
H_{ijk} : &x_i+y_j+z_k=0 &(1\leq i\leq \ell_1, 1\leq j\leq \ell_2, 1\leq k\leq \ell_3)\\
H^{x}_{i_1, i_2} : &x_{i_1}-x_{i_2}=0&(1\leq i_1<i_2\leq \ell_1)\\
H^{y}_{j_1, j_2} : &y_{j_1}-y_{j_2}=0&(1\leq j_1<j_2\leq \ell_2)\\
H^{z}_{k_1, k_2} : &z_{k_1}-z_{k_2}=0&(1\leq k_1<k_2\leq \ell_3)
\end{array}}.\]
Then, 
\[\chi_{\calA_{\ell_1, \ell_2, \ell_3}}(t)=t^2\chi_{\calH_{A_{\ell_1, \ell_2, \ell_3}}}(t).\]
In particular, $r(\calA_{\ell_1, \ell_2, \ell_3})=r(\calH_{A_{\ell_1, \ell_2, \ell_3}})$.
\end{lemma}
\begin{proof}
By Theorem \ref{thm:Zas}, it is enough to show $\chi_{\calA_{\ell_1, \ell_2, \ell_3}}(t)=t^2\chi_{\calH_{A_{\ell_1, \ell_2, \ell_3}}}(t)$.
Using the finite field method, we reduce this to show the following equality for a large prime $p$.
\[\left|\F_p^{\ell_1+\ell_2+\ell_3}\setminus\bigcup_{H\in\calA_{\ell_1, \ell_2, \ell_3}}{\overline{H}}\right|=p^2\left|\F_p^{\ell_1+\ell_2+\ell_3-2}\setminus\bigcup_{H\in\calH_{A_{\ell_1, \ell_2, \ell_3}}}{\overline{H}}\right|.\]
Now, we have the following. 
\begin{align*}&\left|\F_p^{d}\setminus\bigcup_{H\in\calA_{\ell_1, \ell_2, \ell_3}}{\overline{H}}\right|=\sum_{(x_{\ell_1}, y_{\ell_2})\in\F_p^2}{\#\Set{(\bm{x}, \bm{y}, \bm{z}) \in \F_p^{d} \ | \ 
\begin{array}{l}
\bullet \ x_{i_1}\neq x_{i_2} \ (1\leq i_1< i_2\leq \ell_1)\\
\bullet \ y_{j_1}\neq y_{j_2} \ (1\leq j_1< j_2\leq \ell_2)\\
\bullet \ z_{k_1}\neq z_{k_2} \ (1\leq k_1< k_2\leq \ell_3)\\
\bullet \ x_i+y_j+z_k\neq0 \ \left(\begin{array}{l}1\leq i\leq \ell_1,\\ 1\leq j\leq \ell_2, \\ 1\leq k\leq \ell_3\end{array}\right)\end{array}}},\\
\end{align*}
where $d=\ell_1+\ell_2+\ell_3-2$, $\bm{x}:=(x_1, \ldots ,x_{\ell_1-1})$, $\bm{y}:=(y_1, \ldots ,y_{\ell_2-1})$, $\bm{z}:=(z_1, \ldots ,z_{\ell_3})$.
By changing of the coordinate as $x'_i=x_i-x_{\ell_1} \ (1\leq i\leq\ell_1-1)$, $y'_j=y_j-y_{\ell_2} \ (1\leq j\leq\ell_2-1)$, $z'_k=z_k+x_{\ell_1}+y_{\ell_2} \ (1\leq k\leq\ell_3)$, we can prove that the sum above is equal to the following. 
\begin{align*}
&\sum_{(x_{\ell_1}, y_{\ell_2})\in\F_p^2}{\#\Set{(\bm{x'}, \bm{y'}, \bm{z'}) \in \F_p^{d} \ | \ \begin{array}{l}
\bullet \ x'_i\neq0 \ (1\leq i\leq \ell_1-1)\\
\bullet \ y'_j\neq0 \ (1\leq j\leq \ell_2-1)\\
\bullet \ z'_k\neq0 \ (1\leq k\leq \ell_3)\\
\bullet \ x'_i+z'_k\neq0 \ (1\leq i\leq \ell_1-1, \ 1\leq k\leq \ell_3)\\
\bullet \ y'_j+z'_k\neq0 \ (1\leq j\leq \ell_2-1, \ 1\leq k\leq \ell_3)  \\
\bullet \ x'_i+y'_j+z'_k\neq0 \ \left(\begin{array}{l}1\leq i\leq \ell_1-1,\\ 1\leq j\leq \ell_2-1, \\ 1\leq k\leq \ell_3\end{array}\right)\\
\bullet \ x'_{i_1}\neq x'_{i_2} \ (1\leq i_1<i_2\leq \ell_1-1)\\
\bullet \ y'_{j_1}\neq y'_{j_2} \ (1\leq j_1<j_2\leq \ell_2-1)\\
\bullet \ z'_{k_1}\neq z'_{k_2} \ (1\leq k_1<k_2\leq \ell_3)\\
\end{array}}}\\
&=p^2\left|\F_p^{d}\setminus\bigcup_{H\in\calH_{A_{\ell_1, \ell_2, \ell_3}}}{\overline{H}}\right|.
\end{align*}
\end{proof}
In the case of $\ell_3=1 \ \text{or} \ 2$, the characteristic polynomial of $\calA_{\ell_1, \ell_2, \ell_3}$ is known as the following by Edelman and Reiner. 
\begin{theorem}{\rm (Edelman and Reiner \cite[Theorem 2.5(2)]{ER})}\\
\vspace{-8pt}
\begin{itemize}
\item[(1)]When $\ell_3=1$, 
\[\chi_{\calA_{\ell_1, \ell_2, 1}}(t)=t^2(t-1)(t-2)\cdots(t-(\ell_1+\ell_2-1)).\]
\[r(\calA_{\ell_1, \ell_2, 1})=(\ell_1+\ell_2)!.\]
\item[(2)]When $\ell_3=2$, 
\[\chi_{\calA_{\ell_1, \ell_2, 2}}(t)=t^2(t-1)\displaystyle\prod_{i=\ell_1+1}^{\ell_1+\ell_2}{(t-i)}\displaystyle\prod_{j=\ell_2+1}^{\ell_1+\ell_2-1}{(t-j)}.\]
\[r(\calA_{\ell_1, \ell_2, 2})=\fr<2\binom{\ell_1+\ell_2+1}{\ell_1}\binom{\ell_1+\ell_2+1}{\ell_2}\ell_1!\ell_2!/\ell_1+\ell_2+1>.\]
\end{itemize}
\end{theorem}

\begin{remark}\hspace{2pt}\\
\begin{itemize}
\vspace{-5mm}
\item[(1)]Furthermore, in \cite[Theorem 2.5(2)]{ER}, they showed that $\calA_{\ell_1, \ell_2, 1}$ and $\calA_{\ell_1, \ell_2, 2}$ are {\it free} arrangements, and they also determined  their {\it exponents}. By the general theory of hyperplane arrangements, the characteristic polynomial of a free hyperplane arrangement $\calA\subset\R^d$ with exponents $\{e_1, \ldots ,e_d\}$ is
\[\chi_\calA(t)=\prod_{i=1}^d{(t-e_i)}.\]
\item[(2)]In \cite{ER}, they also showed if
$\ell_1, \ell_2, \ell_3\geq3$, then $\calA$ will not be a free arrangement. 
Actually, the characteristic polynomial of $\calA_{3,3,3}$ is the following, 
\[\chi_{A_{3,3,3}}(t)=t^2(t-1)(t-5)(t-7)(t^4-23t^3+200t^2-784t+1188).\]
\end{itemize}
\end{remark}
Using above results, we can compute the number of crepant resolutions of 4-dimensional affine hypertoric varieties $Y(A_{\ell_1, \ell_2, \ell_3}, 0)=\overline{\calO^{\tmin}}(\{\ell_1, \ell_2, \ell_3\})$ in the case of $\ell_3=1$ or 2.  
\begin{corollary}\hspace{2pt}\label{cor:counting_4dim}\\
\begin{itemize}
\vspace{-5mm}
\item[(1)]The number of crepant resolutions of $\overline{\calO^{\tmin}}(\{\ell_1, \ell_2, 1\})$ is 
\[\binom{\ell_1+\ell_2}{\ell_1}.\]
\item[(2)]The number of crepant resolutions of $\overline{\calO^{\tmin}}(\{\ell_1, \ell_2, 2\})$ is
\[\fr<\binom{\ell_1+\ell_2+1}{\ell_1}\binom{\ell_1+\ell_2+1}{\ell_2}/\ell_1+\ell_2+1>.\]
\end{itemize}
\end{corollary}

On the other hand, we don't know the number of all crepant resolutions of $\overline{\calO^{\tmin}}(\{\ell_1, \ell_2, \ell_3\})$ in general case. 
\begin{question}
When $\ell_1, \ell_2, \ell_3\geq3$, compute the number of crepant resolutions of $\overline{\calO^{\tmin}}(\{\ell_1, \ell_2, \ell_3\})$. 
\end{question}


\end{document}